\documentclass{svjour3}                     

\smartqed  

\makeatletter
\renewcommand*{\@thmcounterend}{}
\makeatother
\spnewtheorem{assumption}{Assumption}{\bfseries}{\itshape}

\usepackage{cite}

\usepackage{amsmath, amsfonts, amssymb}
\usepackage{hyperref}
\usepackage[titletoc,toc,title]{appendix}
\usepackage{tikz}
\usepackage{graphicx, color, bm}
\usepackage{subfig}
\usepackage{pgfplots}
\usepackage{pgfplotstable}
\definecolor{markercolor}{RGB}{124.9, 255, 160.65}

\renewcommand{\hat}{\widehat}
\renewcommand{\tilde}{\widetilde}

\newcommand*\diff[1]{\mathop{}\!{\mathrm{d}#1}}
\newcommand{\diag}[1]{{\rm diag}\LRp{#1}}
\newcommand{\td}[2]{\frac{{\rm d}#1}{{\rm d}{\rm #2}}}
\newcommand{\pd}[2]{\frac{\partial#1}{\partial#2}}
\newcommand{\nor}[1]{\left\| #1 \right\|}
\newcommand{\LRp}[1]{\left( #1 \right)}
\newcommand{\LRs}[1]{\left[ #1 \right]}

\newcommand{\LRb}[1]{\left| #1 \right|}
\newcommand{\LRc}[1]{\left\{ #1 \right\}}

\newcommand{\Grad} {\ensuremath{\nabla}}

\newcommand{\avg}[1] {\ensuremath{\LRc{\!\{#1\}\!}}}

\renewcommand{\note}[1]{#1}
\newcommand{\bnote}[1]{#1}
\newcommand{\rnote}[1]{#1}
\graphicspath{{./figs/}}

\date{}
\author{Jesse Chan}
\title{Skew-symmetric entropy stable modal discontinuous Galerkin formulations}
\titlerunning{Skew-symmetric entropy stable DG formulations}

\begin{document}

\institute{J. Chan \at 
Department of Computational and Applied Mathematics\\
Rice University\\
              6100 Main Street\\ 
              Houston, TX, 77005 \\
              \email{jesse.chan@rice.edu}
}

\maketitle

\begin{abstract}
High order entropy stable discontinuous Galerkin (DG) methods for nonlinear conservation laws satisfy an inherent discrete entropy inequality.  The construction of such schemes has relied on the use of carefully chosen nodal points \cite{gassner2013skew, fisher2013high, carpenter2014entropy, crean2018entropy, chan2018efficient} or volume and surface quadrature rules \cite{chan2017discretely, chan2018discretely} to produce operators which satisfy a summation-by-parts (SBP) property.  In this work, we show how to construct ``modal'' DG formulations which are \bnote{entropy stable for volume and surface quadratures under which the SBP property in \cite{chan2017discretely} does not hold.  These formulations rely on an alternative skew-symmetric construction of operators which automatically satisfy the SBP property.   Entropy stability then follows for choices of volume and surface quadrature which satisfy sufficient accuracy conditions.}   \rnote{The accuracy of these new SBP operators depends on a separate set of conditions on quadrature accuracy, with design order accuracy recovered under the usual assumptions of degree $2N-1$ volume quadratures and degree $2N$ surface quadratures. }
We conclude with numerical experiments verifying the accuracy and stability of the proposed formulations, and discuss an application of these formulations for entropy stable DG schemes on mixed quadrilateral-triangle meshes.  
\end{abstract}


\section{Introduction}

High order methods for the simulation of time-dependent compressible flow have the potential to achieve higher levels of accuracy at lower costs compared to current low order schemes \cite{wang2013high}.  In addition to superior accuracy, the low numerical dispersion and dissipation of high order methods \cite{ainsworth2004dispersive} enables the accurate propagation of waves over long distances and time scales.  The same properties also make high order methods attractive for unsteady phenomena such as vorticular and turbulent flows, which are sensitive to numerical dissipation \cite{visbal1999high, wang2013high}.  

However, when applied to nonlinear conservation laws, high order methods can experience artificial growth and blow-up near under-resolved features such as shocks or turbulence.  In practice, the application of high order methods to practical problems requires shock capturing and stabilization techniques (such as artificial viscosity) or solution regularization (such as filtering or limiting) to prevent solution blow-up.  The resulting schemes for nonlinear conservation laws walk a fine line between stability, robustness, and accuracy.  Aggressive stabilization or regularization can result in the loss of high order accuracy, while too little can result in instability \cite{wang2013high}.  Moreover, it can be difficult to determine robust expressions for stabilization paramaters, as parameters which work for one simulation can fail when applied to a different physical regime or discretization setting.  

These issues have motivated the introduction of high order \textit{entropy stable} discretizations, which satisfy a semi-discrete entropy inequality while maintaining high order accuracy in smooth regions.  Proofs of continuous entropy inequalities rely on the chain rule, which does not hold discretely due to effects such as quadrature error.   Entropy stable schemes \rnote{were originally introduced in the context of finite volume methods} \cite{tadmor1987numerical, tadmor2003entropy, fjordholm2012arbitrarily,chandrashekar2013kinetic, tadmor2016entropy, ray2016entropy}.  They were then extended to high order collocation DG methods on tensor product elements in \cite{fisher2013high, carpenter2014entropy, gassner2016split, gassner2017br1} and to simplicial elements in \cite{crean2017high, chen2017entropy, crean2018entropy, chan2017discretely, chan2018discretely}.  \rnote{These extensions combine summation-by-parts (SBP) differentiation operators, which satisfy a matrix analogue of integration by parts, with ``flux differencing'' for the discretization of nonlinear convective terms.  Together, these techniques circumvents the loss of the chain rule while preserving a semi-discrete analogue of the continuous entropy inequality.}
Entropy stable methods have also been extended to a variety of other discretization settings, including staggered grids \cite{parsani2016entropy, fernandez2019staggered}, Gauss-Legendre collocation \cite{chan2018efficient}, and non-conforming meshes \cite{friedrich2017entropy}.  

Entropy stable ``modal'' DG discretizations \cite{chan2017discretely, chan2018discretely} are built upon flux differencing and the SBP property.  However, the SBP property does not hold for certain under-integrated quadrature rules, which arise naturally in some discretization settings.  For example, on hybrid meshes consisting of both quadrilateral and triangular elements, it is convenient to utilize the same quadrature rule on shared faces between different element types.  On degree $N$ quadrilateral elements, a popular choice of quadrature is an $(N+1)$-point Gauss-Legendre-Lobatto (GLL) rule.  When both volume and surface integrals are approximated using $(N+1)$ point GLL quadrature rules, the SBP property holds, despite the fact that GLL quadrature is inexact for the integrands which appear in finite element formulations \cite{fisher2013high}.  However, while GLL quadrature induces an SBP property on quadrilateral elements, it does not guarantee an SBP property if used on triangular elements \cite{chan2017discretely}.  

This work proposes an alternative formulation which utilizes a \note{skew-symmetric construction of the SBP operator which satisfies the SBP property by construction}.  Under such a formulation, the proof of entropy stability \note{holds under weaker quadrature rules} compared to the SBP property introduced in \cite{chan2017discretely, chan2018discretely}.  We show that this skew-symmetric formulation is entropy stable, locally conservative, and free-stream preserving on curved elements, and confirm theoretical results with numerical experiments on hybrid triangular-quadrilateral meshes.  

It should be noted that a similar approach to entropy stable discretizations was introduced within a finite difference framework \cite{chen2017entropy, crean2018entropy} using multidimensional differencing operators which satisfy similar accuracy conditions and an SBP property \cite{hicken2016multidimensional}.  These operators exist for nodal points corresponding to sufficiently accurate choices of volume and surface quadrature, but do not correspond to any specific basis or approximation space.  The formulations in \cite{chen2017entropy, crean2018entropy} differ from the ones presented in this work in that they are based on SBP finite differences and ``nodal'' (rather than ``modal'') DG formulations, with differentiation operators computed algebraically or through an optimization problem for each specific choice of nodes.  In contrast, ``modal'' formulations induce quadrature-based operators from an explicit approximation space, and accomodate general choices of volume and surface quadrature (e.g.\ volume quadratures without boundary nodes and over-integrated quadrature rules).  

The structure of the paper is as follows: Section~\ref{sec:nonlin} describes the continuous entropy inequality which we aim to replicate discretely.  Section~\ref{sec:approx} and Section~\ref{sec:sbp} introduce polynomial approximation spaces and quadrature-based SBP operators on simplicial and tensor product elements.  
Section~\ref{sec:skew1} introduces \note{an alternative skew-symmetric construction of SBP operators and describes how to construct entropy stable formulations on a reference element.  Connections between the accuracy of the new skew-symmetric SBP operators and quadrature accuracy are also discussed.}  Section~\ref{sec:skew2} extends the skew-symmetric formulation to curved elements, and provides explicit conditions for entropy stability in terms of quadrature accuracy and the polynomial degree of geometric mappings.  Section~\ref{sec:num} concludes by presenting numerical experiments which verify the theoretical assumptions, stability, and accuracy of the proposed formulations.


\section{Entropy stability for systems of nonlinear conservation laws}
\label{sec:nonlin} 

We begin by reviewing the dissipation of entropy for a $d$-dimensional system of nonlinear conservation laws on a domain $\Omega$
\begin{equation}
\pd{\rnote{\bm{u}}}{t}  + \sum_{j=1}^d\pd{\bm{f}_j(\bm{u})}{x_j} = \bm{0}, \qquad \bm{u}\in \mathbb{R}^n, \qquad \bm{f}:\mathbb{R}^n\rightarrow\mathbb{R}^n,
\label{eq:nonlineqs}
\end{equation}
where $\bm{u}$ are the conservative variables and $\bm{f}(\bm{u})$ is a vector-valued nonlinear flux function.  We are interested in nonlinear conservation laws for which a convex entropy function $U(\bm{u})$ exists.  For such systems, the  \emph{entropy variables} are an invertible mapping $\bm{v}(\bm{u}):\mathbb{R}^n\rightarrow \mathbb{R}^n$ defined as the derivative of the entropy function with respect to the conservative variables 
\begin{align}
\bm{v}(\bm{u}) = \pd{U}{\bm{u}}.
\label{eq:entropyvarsmap}
\end{align}
Several widely used equations in fluid modeling (Burgers, shallow water, compressible Euler and Navier-Stokes equations) yield convex entropy functions $U(\bm{u})$ \cite{hughes1986new, chen2017entropy}.  Let $\partial \Omega$ be the boundary of $\Omega$ with outward unit normal $\bm{n}$.  By multiplying the equation (\ref{eq:nonlineqs}) with $\bm{v}(\bm{u})^T$, the solutions $\bm{u}$ of (\ref{eq:nonlineqs}) can be shown to satisfy an entropy inequality
\begin{equation}
\int_{\Omega}\pd{U(\bm{u})}{t}\diff{x} + \int_{\partial \Omega} \sum_{j=1}^d \LRp{\bm{v}(\bm{u})^T\bm{f}_j(\bm{u}) - \psi_j\LRp{\bm{v}(\bm{u})}}n_j \diff{x} \leq 0, 
\label{eq:entropyineq}
\end{equation}
where $\bm{n} = \LRp{n_1,\ldots,n_d}$ denotes the outward unit normal, and $\psi_j(\bm{u})$ is some function referred to as the entropy potential.  

The proof of (\ref{eq:entropyineq}) requires the use of the chain rule \cite{mock1980systems, harten1983symmetric, dafermos2005compensated}.  The instability-in-practice of high order schemes for (\ref{eq:nonlineqs}) can be attributed in part to the fact that the discrete form of the equations do not satisfy the chain rule, and thus do not satisfy (\ref{eq:entropyineq}).  As a result, discretizations of (\ref{eq:nonlineqs}) do not typically possess an underlying statement of stability.  This can be offset in practice by the numerical dissipation inherent in lower order schemes.  However, because high order discretizations possess low numerical dissipation, the lack of an underlying discrete stability has contributed to the perception that high order methods are inherently less stable than low order methods.

\section{Polynomial approximation spaces}
\label{sec:approx}

In this work, we consider either simplicial reference elements (triangles and tetrahedra) or tensor product reference elements (quadrilaterals and hexahedra).  We define an approximation space using degree $N$ polynomials on the reference element; however, the natural polynomial approximation space differs depending on the element type \cite{chan2015gpu}.  On a $d$-dimensional reference simplex, the natural polynomial space consists of total degree $N$ polynomials 
\[
P^N\LRp{\hat{D}} = \LRc{\hat{x}_1^{i_1}\ldots\hat{x}_d^{i_d}, \quad \hat{\bm{x}} \in \hat{D}, \quad 0\leq \sum_{k=1}^d i_k \leq N}.
\]
In contrast, the natural polynomial space on a $d$-dimensional tensor product element is the space of maximum degree $N$ polynomials
\[
Q^N\LRp{\hat{D}} = \LRc{\hat{x}_1^{i_1}\ldots\hat{x}_d^{i_d}, \quad \hat{\bm{x}} \in \hat{D}, \quad 0\leq i_k \leq N, \quad k = 1,\ldots, d}.
\]
We denote the natural approximation space on a given reference element $\hat{D}$ by $V^N$.  Furthermore, we denote the dimension of $V^N$ as $N_p = {\rm dim}\LRp{V^N\LRp{\hat{D}}}$.  

The proofs presented in this work will also refer to anisotropic tensor product polynomial spaces, where the maximum polynomial degree varies depending on the coordinate direction.  We denote such spaces by $Q^{N_1, \ldots, N_d}$, where $N_k$ are non-negative integers and
\[
Q^{N_1, N_2, \ldots, N_d}\LRp{\hat{D}} = \LRc{\hat{x}_1^{i_1}\ldots\hat{x}_d^{i_d}, \quad \hat{\bm{x}} \in \hat{D}, \quad 0\leq i_k \leq N_k, \quad k = 1,\ldots, d}.
\]
For example, the isotropic tensor product space $Q^N$ is the same as $Q^{N,\ldots,N}$.

We also define trace spaces for each reference element.  Let $\hat{f}$ be a face of the reference element $\hat{D}$.  The trace space $V^N \LRp{\hat{f}}$ is defined as the restrictions of functions in $V^N\LRp{\hat{D}}$ to $\hat{f}$, and denote the dimension of the trace space as ${\rm dim}\LRp{V^N\LRp{{\hat{f}}}} = N^f_p$.  
\[
V^N \LRp{\hat{f}} = \LRc{ \left.u\right|_{\hat{f}}, \quad u \in V^N\LRp{\hat{D}}, \quad \hat{f}\in \partial\hat{D}}.
\]
For example, on a $d$-dimensional simplex, $V^N \LRp{\partial \hat{D}}$ consists of total degree $N$ polynomials on simplices of dimension $(d-1)$.  On a $d$-dimensional tensor product element, $V^N \LRp{\partial \hat{D}}$ consists of maximum degree $N$ polynomials on a tensor product element of dimension $(d-1)$.  



\section{Quadrature-based matrices and ``hybridized'' SBP operators}
\label{sec:sbp}
Let $\hat{D} \subset\mathbb{R}^d$ denote a reference element with surface $\partial \hat{D}$.  
The high order schemes in \cite{chan2017discretely, chan2018discretely} begin by approximating the solution in a degree $N$ polynomial basis $\LRc{\phi_j(\hat{\bm{x}})}_{\rnote{j}=1}^{N_p}$ on $\hat{D}$.  These schemes also assume volume and surface quadrature rules $(\hat{\bm{x}}_i, w_i)$, $\LRp{\hat{\bm{x}}^f_i,w^f_i}$ on $\hat{D}$.  We will specify the accuracy of each quadrature rule later, and discuss how quadrature accuracy implies specific operator properties.  

Let $\bm{V}_q,\bm{V}_f$ denote interpolation matrices, and let $\bm{D}^i$ be the differentiation matrix with respect to the $i$th coordinate such that
\begin{gather}
\LRp{\bm{V}_q}_{ij} = \phi_j(\hat{\bm{x}}_i), \qquad \LRp{\bm{V}_f}_{ij} = \phi_j(\hat{\bm{x}}^f_i), \qquad \pd{\phi_j(\hat{\bm{x}})}{\hat{x}_i} = \sum_{k=1}^{N_p} \LRp{\bm{D}^i_{jk}} \phi_k(\hat{\bm{x}}).
\end{gather}
The interpolation matrices $\bm{V}_q,\bm{V}_f$ map basis coefficients to evaluations at volume and surface quadrature points respectively, while the differentiation matrix ${\bm{D}}_i$ maps basis coefficients of a function to the basis coefficients of its derivative with respect to $\hat{x}_k$.  The interpolation matrices are used to assemble the mass matrix $\bm{M}$, the quadrature-based projection matrix $\bm{P}_q$, and lifting matrix $\bm{L}_f$
\begin{gather}
\bm{M} = \bm{V}_q^T\bm{W}\bm{V}_q, \qquad \bm{P}_q = \bm{M}^{-1}\bm{V}_q^T\bm{W}, \qquad \bm{L}_f = \bm{M}^{-1}\bm{V}_f^T\bm{W}_f,
\end{gather}
where $\bm{W}, \bm{W}_f$ are diagonal matrices of volume and surface quadrature weights, respectively.  \bnote{We have also assumed that the volume quadrature is sufficiently accurate such that the mass matrix $\bm{M}$ is positive-definite and invertible.} The matrix $\bm{P}_q$ is a quadrature-based discretization of the $L^2$ projection operator $\Pi_N$ onto degree $N$ polynomials, which is given as follows: find $\Pi_N u \in V^N$ such that
\begin{equation}
\int_{\hat{D}} \Pi_N u v = \int_{\hat{D}} u v, \qquad \forall v \in V^N.
\label{eq:l2proj}
\end{equation}

Interpolation, differentiation, and $L^2$ projection matrices can be combined to construct finite difference operators.  For example, the matrix $\bm{D}^i_q = \bm{V}_q\bm{D}^i\bm{P}_q$ maps function values at quadrature points to approximate values of the derivative at quadrature points.  By choosing specific quadrature rules, $\bm{D}^i_q$ recovers high order summation-by-parts finite difference operators in \cite{gassner2013skew, fernandez2014generalized, ranocha2018generalised} and certain operators in \cite{hicken2016multidimensional}.  However, to address difficulties in designing efficient entropy stable interface terms for nonlinear conservation laws, a new ``hybridized'' summation by parts matrix was introduced in \cite{chan2017discretely} which builds interface terms directly into the approximation of the derivative.\footnote{\note{The term ``hybridized'' SBP operator was introduced in the review paper \cite{chenreview}.  These operators were originally referred to as ``decoupled'' SBP operators in \cite{chan2017discretely})}.}  

Let $\hat{\bm{n}}$ denote the scaled outward normal vector $\hat{\bm{n}} = \LRc{\hat{n}_1\hat{J}_f,\ldots,\hat{n}_d\hat{J}_f}$, where $\hat{J}_f$ is the determinant of the Jacobian of the mapping of a face of $\partial \hat{D}$ to a reference face.  Let $\hat{\bm{n}}_i$ denote the vector containing values of the $i$th component $\hat{n}_i\hat{J}_f$ at all surface quadrature points, and define the generalized SBP operator 
\[
\bm{Q}^i = \bm{W}\bm{D}^i_q = \bm{W}\bm{V}_q\bm{D}^i\bm{P}_q.
\]
The ``hybridized'' summation by parts operator $\bm{Q}^i_N$ is defined as the block matrix involving both volume and surface quadratures
\begin{gather}
\bm{E} = \bm{V}_f\bm{P}_q, \qquad \bm{B}^i = \bm{W}_f \diag{\hat{\bm{n}}_i}, \qquad \bm{Q}^i_N  = \LRs{
\begin{array}{cc}
\bm{Q}^i - \frac{1}{2}\bm{E}^T\bm{B}^i\bm{E} &  \frac{1}{2}\bm{E}^T\bm{B}^i\\
-\frac{1}{2}\bm{B}^i\bm{E} & \frac{1}{2} \bm{B}^i
\end{array}}.  \label{eq:QN}
\end{gather}
Here, $\bm{B}^i$ is a boundary ``integration'' matrix, and $\bm{E}$ denotes the extrapolation matrix which maps values at volume quadrature points to values at surface quadrature points using quadrature-based $L^2$ projection and polynomial interpolation.  


For $\bm{Q}^i$ which satisfy a ``generalized'' SBP property, the matrix $\bm{Q}^i_N$ also satisfies a summation-by-parts (SBP) property, which is used to prove semi-discrete entropy stability for nonlinear conservation laws.  
\begin{theorem}
If $\bm{Q}^i$ satisfies the generalized SBP property
\begin{gather}
\bm{Q}^i = \bm{E}^T\bm{B}^i\bm{E} - \LRp{\bm{Q}^i}^T,
\label{eq:gsbp}
\end{gather}
then the hybridized SBP operator $\bm{Q}^i_N$ (\ref{eq:QN}) satisfies a summation by parts property:
\begin{gather}
\bm{Q}^i_N+\LRp{\bm{Q}^i_N}^T = \bm{B}^i_N, \qquad \bm{B}^i_N = \LRp{\begin{array}{cc}\bm{0}&\\ & \bm{B}^i\end{array}}.\label{eq:dsbp}
\end{gather}
\label{lemma:dsbp}
\end{theorem}
\begin{proof}
The proof is a straightforward extension of Theorem 1 in \cite{chan2017discretely} to polynomial approximation spaces on non-simplicial elements.  
\qed\end{proof}

The matrix $\bm{Q}^i$ satisfies a generalized SBP property if the volume and surface quadrature rules are sufficiently accurate such that the quantities
\[
\int_{\hat{D}} \pd{u}{\hat{x}_i} v, \qquad \int_{\partial \hat{D}} u v \hat{n}_i
\]
are integrated exactly for all $u,v \in V^N\LRp{\hat{D}}$ and $i = 1,\ldots, d$.  This implies that Theorem~\ref{lemma:dsbp} is satisfied for sufficiently accurate volume and surface quadratures.  For example, on simplicial elements, (\ref{eq:dsbp}) holds if the volume quadrature is exact for polynomial integrands of total degree $(2N-1)$, and the surface integral is exact for degree $2N$ polynomials on each face.  Tensor product elements require stricter conditions: (\ref{eq:dsbp}) holds if both the volume and surface quadratures are exact for polynomial integrands of degree $2N$ in each coordinate, due to the fact that derivatives of $u\in Q^N$ are degree $(N-1)$ polynomials with respect to one coordinate and degree $N$ with respect to others.  

\begin{remark}
It should be stressed that the accuracy conditions on volume and surface quadratures are sufficient but not necessary conditions for Theorem~\ref{lemma:dsbp}.  For example, it is well known that the use of $(N+1)$ point Gauss-Legendre-Lobatto (GLL) rules for both volume and surface quadratures result in a generalized SBP property, despite the fact that these rules are only accurate for degree $(2N-1)$ polynomials.  
\label{remark:dsbp}
\end{remark}

When a generalized SBP property holds for $\bm{Q}^i$, entropy stability can be proven using the SBP property in Theorem~\ref{lemma:dsbp} \cite{chan2017discretely, chan2018discretely}.  The focus of this work is to address cases where the generalized SBP property (and as a result, the SBP property in Theorem~\ref{lemma:dsbp}) do not hold. 

\section{Skew-symmetric entropy conservative formulations on a single element}
\label{sec:skew1}
While the SBP property has been used to derive entropy stable schemes,  \bnote{it is difficult to enforce the SBP property (\ref{eq:dsbp}) for $\bm{Q}^i_N$} in certain discretization settings, such as hybrid and non-conforming meshes.  This difficulty is a result of the choices of volume and surface quadrature which naturally arise in these settings.  We first illustrate how specific pairings of volume and surface quadratures can result in the loss of the SBP property (\ref{eq:dsbp}) \bnote{for $\bm{Q}^i_N$}.  We then propose an \note{alternative skew-symmetric version of the hybridized SBP operator which satisfies the SBP property by construction}.  The use of these operators results in formulations which are entropy conservative \bnote{under a wider range of quadratures.}

\subsection{Loss of the SBP property}

In this section, we give examples of specific pairings of volume and surface quadratures under which the decoupled SBP property does not hold (see Figure~\ref{fig:sbploss}).  We consider two dimensional reference elements $\hat{D}$ with spatial coordinates $x,y$.
\begin{figure}
\centering
\subfloat[GLL volume quadrature, Gauss surface quadrature]{\includegraphics[width=.4\textwidth]{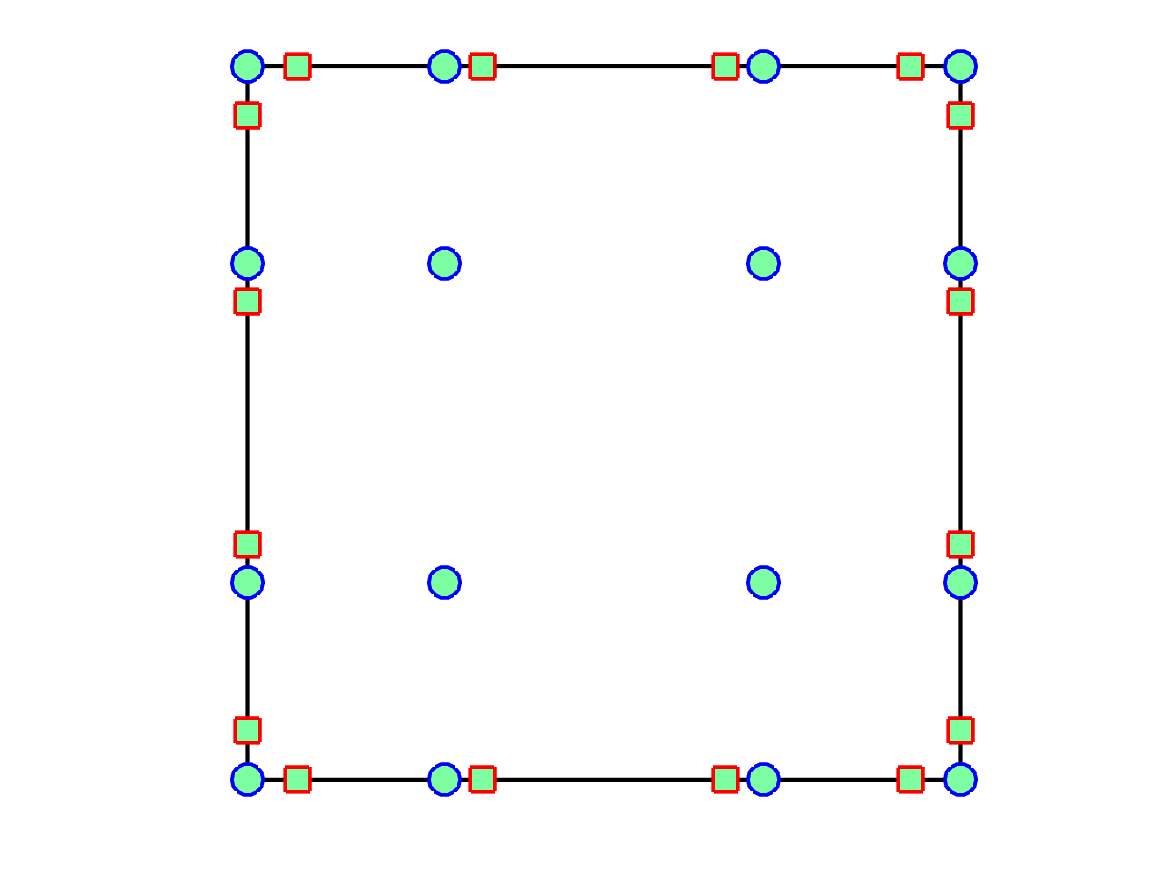}\label{subfig:gllgq}}
\hspace{2em}
\subfloat[Degree $2N$ volume quadrature, GLL surface quadrature]{\includegraphics[width=.4\textwidth]{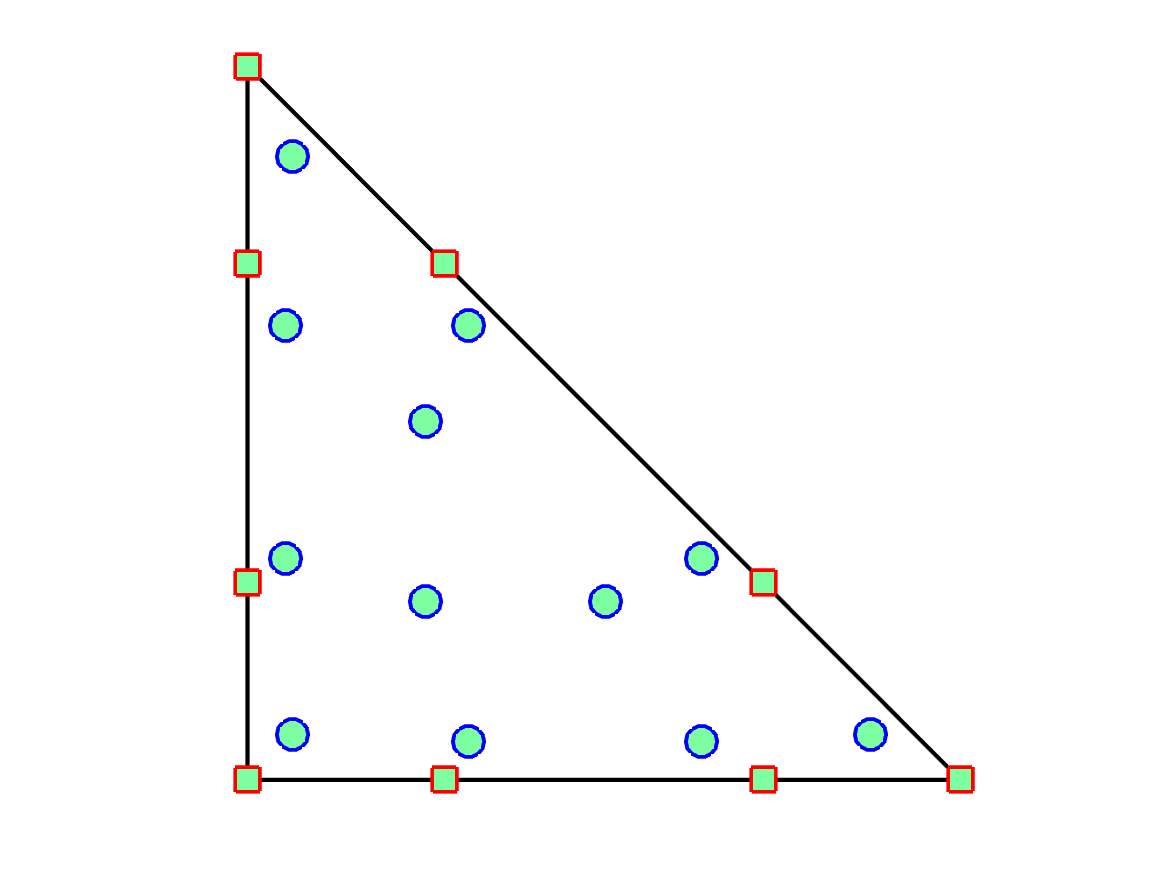}\label{subfig:trigll}}
\caption{Volume and surface quadrature pairs which do not satisfy the assumptions of Theorem~\ref{lemma:dsbp}, and thus do not possess the decoupled SBP property (\ref{eq:dsbp}). Volume quadrature nodes are drawn as circles, while surface quadrature nodes are drawn as squares.}
\label{fig:sbploss}
\end{figure}

\paragraph{Quadrilateral elements (Figure~\ref{subfig:gllgq})} We first consider a quadrilateral element $\hat{D}$ with an $(N+1)$ point tensor product GLL volume quadrature and $(N+1)$ point Gauss quadrature on each face.  Let $u,v \in Q^N$ denote two arbitrary degree $N$ polynomials.  The assumptions of Theorem~\ref{lemma:dsbp} are that the volume quadrature exactly integrates $\int_{\hat{D}} \pd{u}{x_i} v$ and that the surface quadrature exactly integrates $\int_{\partial \hat{D}} u v \hat{n}_i$ on $\hat{D}$.  Because the $(N+1)$-point Gauss rule is exact for polynomials of degree $2N+1$ and the product $uv \in P^{2N}$ on each face, the surface quadrature satisfies the assumptions of Theorem~\ref{lemma:dsbp}.  However, the 1D GLL rule is only exact for polynomials of degree $(2N-1)$.  The derivative $\pd{u}{x}$ is a polynomial of degree $(N-1)$ in $x$, but is degree $N$ in $y$.  Thus, $\pd{u}{x}v$ is a polynomial of degree $(2N-1)$ in $x$ but degree $2N$ in $y$, and is not integrated exactly by the volume quadrature.  

\paragraph{Triangular elements (Figure~\ref{subfig:trigll})} We next consider a triangular element $\hat{D}$, where the volume quadrature is exact for degree $2N$ polynomials \cite{xiao2010quadrature} and an $(N+1)$-point GLL quadrature on each face.  Let $u,v \in P^N$ denote two arbitrary degree $N$ polynomials.  The derivative $\pd{u}{x} \in P^{(N-1)}$, and $\pd{u}{x}v \in P^{(2N-1)}$, so the volume quadrature satisfies the assumptions of Theorem~\ref{lemma:dsbp}.  However, because the surface quadrature is exact only degree $(2N-1)$ polynomials and the trace of $uv\in P^{2N}$, the surface quadrature does not satisfy the assumptions of Theorem~\ref{lemma:dsbp}.

\begin{figure}
\centering
\begingroup
\captionsetup[subfigure]{width=.425\textwidth}
\subfloat[Insufficiently accurate surface quadrature on the triangle element.]{\includegraphics[width=.425\textwidth]{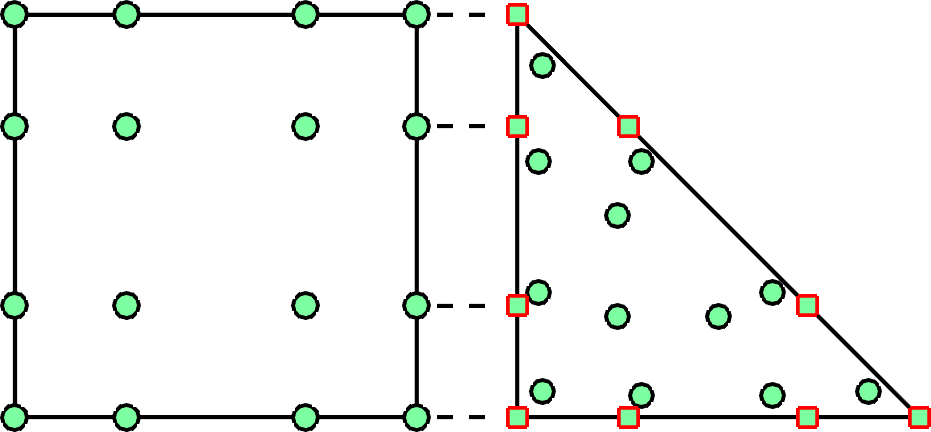}\label{subfig:hybrid1}}
\endgroup
\hspace{2em}
\subfloat[Incompatible surface quadrature on the quadrilateral element.]{\includegraphics[width=.425\textwidth]{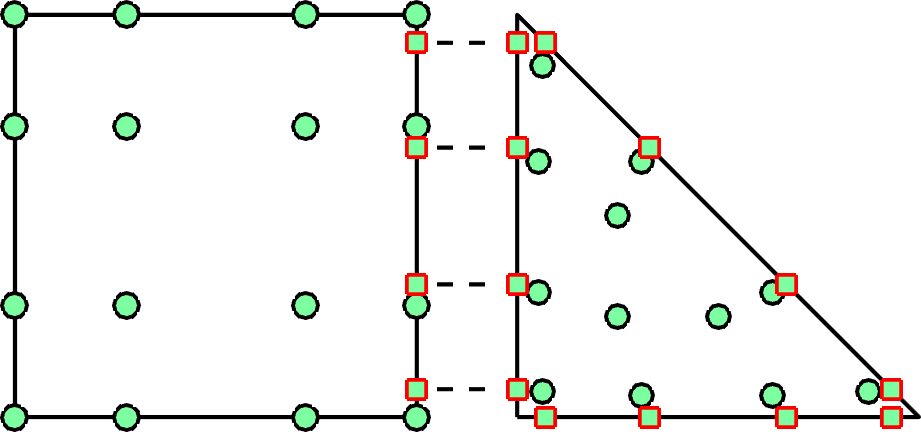}\label{subfig:hybrid2}}
\caption{Examples of interface couplings which do not result in a decoupled SBP property (\ref{eq:dsbp}). Volume quadrature nodes are drawn as circles, while surface quadrature nodes are drawn as  squares. }
\label{fig:hybrid}
\end{figure}

\paragraph{}These specific pairings of volume and surface quadratures appear naturally for hybrid meshes consisting of DG-SEM quadrilateral elements (using GLL volume quadrature) and triangular elements, as shown in Figure~\ref{fig:hybrid}.  In Figure~\ref{subfig:hybrid1}, the surface quadrature is a $(N+1)$ point GLL rule, and results in a loss of the SBP property on the triangle.  In Figure~\ref{subfig:hybrid2}, the surface quadrature is a $(N+1)$ point Gauss-Legendre rule, and results in a loss of the SBP property on the quadrilateral element.  The goal of this work is to construct high order accurate discretizations which preserve entropy conservation for situations in which the decoupled SBP property (\ref{eq:dsbp}) does not hold.  

\subsection{\bnote{An alternative construction of hybridized SBP operators}} 

The property (\ref{eq:dsbp}) relates the polynomial exactness of specific quadrature rules to algebraic properties of quadrature-based matrices.  We will relax accuracy conditions on these quadrature rules by utilizing \note{an alternative construction of $\bm{Q}^i_N$ based on the skew-symmetric matrix $\bm{Q}^i - \LRp{\bm{Q}^i}^T$.  

\begin{lemma}
\bnote{Let $\tilde{\bm{Q}}^i_{N}$ denote the skew-hybridized SBP operator defined by
\begin{equation}
\tilde{\bm{Q}}^i_{N} = \frac{1}{2}\begin{bmatrix}
\bm{Q}^i - \LRp{\bm{Q}^i}^T & \bm{E}^T\bm{B}^i\\
-\bm{B}^i\bm{E} & \bm{B}^i
\end{bmatrix}.
\label{eq:skewQN} 
\end{equation}
Then, $\tilde{\bm{Q}}^i_{N}$ satisfies the SBP property (\ref{eq:dsbp}), and $\tilde{\bm{Q}}^i_N$ and ${\bm{Q}}^i_N$ are identical if $\bm{Q}^i$ satisfies a generalized SBP property (\ref{eq:gsbp}).}
\end{lemma}
\begin{proof}
The SBP property (\ref{eq:dsbp}) holds by construction.  The equivalence between $\tilde{\bm{Q}}^i_N$ and ${\bm{Q}}^i_N$ requires that
\[
\frac{1}{2}\LRp{\bm{Q}^i + \LRp{\bm{Q}^i}^T} = \bm{Q}^i - \frac{1}{2}\bm{E}^T\bm{B}^i\bm{E}.
\]
Rearranging terms shows that this condition is equivalent to a scaling of the GSBP property (\ref{eq:gsbp})
\[
\frac{1}{2}\bm{Q}^i = \frac{1}{2}\LRp{\bm{E}^T\bm{B}^i\bm{E} - \LRp{\bm{Q}^i}^T}.
\]
\end{proof}
While $\tilde{\bm{Q}}^i_N$ is guaranteed to satisfy the SBP property, the accuracy of $\tilde{\bm{Q}}^i_N$ as a differentiation operator now depends on the volume and surface quadrature rules.  Before analyzing accuracy, we first derive conditions under which it is possible to use $\tilde{\bm{Q}}^i_N$ to construct entropy stable formulations of nonlinear conservation laws. }

\subsection{Entropy stability on a reference element}
\label{sec:singleelem}

In this section, we construct so-called ``entropy stable'' schemes on the reference element $\hat{D}$.  These methods ensure that the entropy inequality (\ref{eq:entropyineq}) is satisfied discretely by avoiding the use of the chain rule in the proof of entropy dissipation.  Entropy stable schemes rely on two main ingredients: an entropy stable numerical flux as defined by Tadmor \cite{tadmor1987numerical} and a concept referred to as ``flux differencing''.  Let $\bm{f}_S\LRp{\bm{u}_L,\bm{u}_R}$ be a numerical flux function which is a function of ``left'' and ``right'' states $\bm{u}_L,\bm{u}_R$.  The numerical flux $\bm{f}_S$ is \textit{entropy conservative} if it satisfies the following three conditions:  
\begin{gather}
\bm{f}^i_S(\bm{u},\bm{u}) = \bm{f}_i(\bm{u}), \qquad \text{(consistency)}\label{eq:esflux}\\
\bm{f}^i_S(\bm{u}_L,\bm{u}_R) = \bm{f}^i_S(\bm{u}_R,\bm{u}_R), \qquad \text{(symmetry)}\nonumber\\
\LRp{\bm{v}_L-\bm{v}_R}^T\bm{f}^i_S(\bm{u}_L,\bm{u}_R) = \psi_i(\bm{u}_L) - \psi_i(\bm{u}_R), \qquad \text{(conservation)}\nonumber
\end{gather}
for $i = 1,\ldots, d$.  The construction of entropy stable schemes will utilize \rnote{(\ref{eq:esflux})} in discretizations of both volume and surface terms in a DG formulation.  

We can now construct a skew-symmetric formulation on the reference element $\hat{D}$ and show that it is semi-discretely entropy conservative \bnote{under one additional condition on $\tilde{\bm{Q}}^i_N$}. This formulation can be made entropy stable by adding interface dissipation.  Let $\bm{u}_h$ denote the discrete solution, and let $\bm{u}_q$ denote the values of the solution at volume quadrature points.  We define the auxiliary conservative variables $\tilde{\bm{u}}$ in terms of the $L^2$ projections of the entropy variables 
\begin{gather}
\bm{v}_q = \bm{v}\LRp{\bm{u}_q}, \qquad \tilde{\bm{v}} = \begin{bmatrix}
\bm{V}_q\\
\bm{V}_f
\end{bmatrix}\bm{P}_q\bm{v}_q, \qquad \tilde{\bm{u}} = \bm{u}\LRp{\tilde{\bm{v}}}.
\end{gather}
A matrix formulation for (\ref{eq:nonlineqs}) on $\hat{D}$ is given in terms of $\tilde{\bm{u}}$
\bnote{\begin{gather}
\bm{M}\td{\bm{u}_h}{t} + \sum_{i=1}^d\LRs{\begin{array}{c}
\bm{V}_q \\ \bm{V}_f\end{array}}^T 
\LRp{2\tilde{\bm{Q}}^i_N \circ \bm{F}^i_S}\bm{1} + \bm{V}_f^T\bm{B}^i\LRp{\bm{f}_i^*-\bm{f}(\tilde{\bm{u}}_f)} = 0,  \label{eq:esdgSkew}\\
\LRp{\bm{F}^i_S}_{jk} = \bm{f}^i_S\LRp{\tilde{\bm{u}}_j,\tilde{\bm{u}}_k}, \qquad 1\leq j,k \leq N_q + N^f_q,\nonumber
\end{gather}
where $\tilde{\bm{u}}_f$ denotes the values of $\tilde{\bm{u}}$ on face nodes and $\bm{f}^*$ is some numerical flux, and $N_q, N^f_q$ denote the number of volume and face quadrature points, respectively.	.  This formulation is identical to that of \cite{chan2017discretely}, except that the hybridized SBP operators $\bm{Q}^i_N$ are replaced with their skew-hybridized versions $\tilde{\bm{Q}}^i_N$.  For this reason, we refer to (\ref{eq:esdgSkew}) as the ``skew-symmetric'' formulation.  Under the condition that $\tilde{\bm{Q}}^i_N \bm{1} = \bm{0}$, the formulation (\ref{eq:esdgSkew}) is entropy conservative over $\hat{D}$:}
\begin{theorem}
\bnote{Assume that $\tilde{\bm{Q}}^i_N \bm{1} = \bm{0}$.  }
Then, the formulation (\ref{eq:esdgSkew}) is entropy conservative such that
\begin{equation}
\bm{1}^T\bm{W}\td{U(\bm{u}_q)}{t} + \sum_{i=1}^d\bm{1}^T\bm{B}^i \LRp{\tilde{\bm{v}}_f^T\bm{f}_i^* - \psi_i(\tilde{\bm{u}}_f)} = 0, \qquad \bm{u}_q = \bm{V}_q\bm{u}_h.
\label{eq:esdgthm}
\end{equation}
\rnote{Here, $\psi_i(\tilde{\bm{u}}_f)$ denotes the function $\psi_i$ evaluated at the face values of the entropy-projected conservative variables $\tilde{\bm{u}}_f$}.
\label{thm:esdg}
\end{theorem}
The steps of the proof are identical to those of Theorem 2 in \cite{chan2017discretely}, and we skip them for brevity.

\begin{remark}
We note that (\ref{eq:esdgSkew}) is also equivalent to the following skew-symmetric formulation:
\begin{gather}
\bm{M}\td{\bm{u}_h}{t} + \sum_{i=1}^d\LRs{\begin{array}{c}
\bm{V}_q \\ \bm{V}_f\end{array}}^T 
\LRp{\LRp{\bm{Q}^i_N - \LRp{\bm{Q}^i_N}^T} \circ \bm{F}^i_S}\bm{1} + \bm{V}_f^T\bm{B}^i\bm{f}_i^* = 0,  \label{eq:esdgSkew2}\\
\LRp{\bm{F}^i_S}_{jk} = \bm{f}^i_S\LRp{\tilde{\bm{u}}_j,\tilde{\bm{u}}_k}, \qquad 1\leq j,k \leq N_q + N^f_q,\nonumber
\end{gather}
where the skew-symmetric matrix $\LRp{\bm{Q}^i_N - \LRp{\bm{Q}^i_N}^T}$ possesses the following block structure:
\[
\LRp{\bm{Q}^i_N - \LRp{\bm{Q}^i_N}^T} = \begin{pmatrix}
\bm{Q}_i-\bm{Q}_i^T & {\bm{E}}^T \bm{B}^i\\
-\bm{B}^i\bm{E} & \bm{0}
\end{pmatrix}.
\]
\end{remark}

The skew symmetric formulation can also be shown to be locally conservative in the sense of \cite{shi2017local}, which is necessary to prove that the numerical solution convergences to the weak solution under mesh refinement.  
\begin{theorem}
The formulation (\ref{eq:esdgSkew}) is locally conservative such that
\begin{align}
\bm{1}^T\bm{W}\td{\LRp{\bm{V}_q\bm{u}}}{t} + \sum_{i=1}^d\bm{1}^T\bm{B}^i\bm{f}_i^* = 0. 
\end{align}
\end{theorem}
\begin{proof}
To show local conservation, we test (\ref{eq:esdgSkew}) with $1$
\begin{align}
\bm{1}^T\bm{W}\bm{V}_q\td{\bm{u}_h}{t} + \sum_{i=1}^d
\bm{1}^T
\LRp{\LRp{\bm{Q}^i_N - \LRp{\bm{Q}^i_N}^T} \circ \bm{F}_S}\bm{1} + \bm{1}^T\bm{W}_f \diag{\hat{\bm{n}}}\bm{f}_i^* = 0. 
\end{align}
Because $\bm{F}_S$ is symmetric and $\LRp{\bm{Q}^i_N - \LRp{\bm{Q}^i_N}^T}$ is skew-symmetric, their Hadamard product is also skew-symmetric.  Using that $\bm{x}^T\bm{A}\bm{x} = 0$ for any skew symmetric matrix $\bm{A}$, the volume term $\bm{1}^T\LRp{\LRp{\bm{Q}^i_N - \LRp{\bm{Q}^i_N}^T} \circ \bm{F}_S}\bm{1}$ vanishes.
\qed\end{proof}

\subsection{\bnote{Properties of $\tilde{\bm{Q}}^i_N$ and quadrature accuracy}}

\bnote{The proof of the semi-discrete entropy inequality in Theorem~\ref{thm:esdg} requires both the SBP condition (\ref{eq:dsbp}) and that $\tilde{\bm{Q}}^i_N\bm{1} = \bm{0}$.  While the SBP condition is guaranteed by construction, $\tilde{\bm{Q}}^i_N\bm{1} = \bm{0}$ only holds under sufficiently accurate quadrature rules.  In \cite{chan2017discretely}, it was shown that the hybridized SBP operator satisfies $\bm{Q}^i_N\bm{1} = \bm{0}$ for any volume quadrature such that the mass matrix $\bm{M}$ is positive-definite.  However, ensuring that the skew-hybridized SBP operator satisfies $\tilde{\bm{Q}}^i_N\bm{1} = \bm{0}$ now requires conditions on both volume and surface quadratures which are related to a weak version of the generalized SBP condition (\ref{eq:gsbp}).  }

Throughout the remainder of this work, we will assume that the \bnote{volume and surface quadrature satisfy the following assumptions} for specific functions $v(\bm{x})$:
\begin{assumption}
\label{ass:quad}
Let $v \in V^{N}$ denote some fixed polynomial.  We assume that: 
\begin{enumerate}
\item the mass matrix $\bm{M}$ is positive definite under the volume quadrature rule,
\item the volume quadrature rule is exact for integrals of the form\\$\int_{\hat{D}} \pd{u}{\hat{x}_j} v$ for all $u \in V^N\LRp{\hat{D}}$, $j = 1,\ldots, d$.
\item the surface quadrature rule is exact for integrals of the form\\$\int_{\partial \hat{D}} u v \hat{n}_j$ for all $u \in V^N\LRp{\hat{D}}$, $j = 1,\ldots, d$, and $f \in \partial \hat{D}$.  
\end{enumerate}
\end{assumption}

The conditions of Assumption~\ref{ass:quad} are relatively standard within the SBP literature \cite{hicken2016multidimensional, chan2017discretely, crean2018entropy}, though they have not previously depended on the specific choice of polynomial $v(\bm{x})$.  \bnote{The following theorem shows how these accuracy conditions are related to the condition $\tilde{\bm{Q}}^i_N\bm{1} = \bm{0}$.}
\begin{lemma}
\label{lemma:sbpcor}
\bnote{Suppose Assumption~\ref{ass:quad} holds for $v(\bm{x}) = 1$.  Then, $\tilde{\bm{Q}}^i_N\bm{1} = \bm{0}$.}
\end{lemma}
\begin{proof}
\bnote{
Expanding out $\tilde{\bm{Q}}^i_N\bm{1}$ yields
\[
\tilde{\bm{Q}}^i_N\bm{1} = \frac{1}{2}\begin{bmatrix}
\bm{Q}^i\bm{1} - \LRp{\bm{Q}^i}^T \bm{1} + \bm{E}^T\bm{B}^i \bm{1}\\
-\bm{B}^i\bm{E}\bm{1} + \bm{B}^i\bm{1}\end{bmatrix}.
\]
Here, $\bm{1}$ denotes the appropriate length vector with all entries equal to one.  
Since polynomials are equal to their $L^2$ projection, $\bm{E}\bm{1} = \bm{1}$ \cite{chan2017discretely,chan2018discretely}, and 
\[
-\bm{B}^i\bm{E}\bm{1} + \bm{B}^i\bm{1} = \bm{0}.  
\]
Moreover, since $\bm{Q}^i$ is a differentation matrix, $\bm{Q}^i\bm{1} = \bm{0}$, and showing $\tilde{\bm{Q}}^i_N\bm{1} = \bm{0}$ reduces to showing that 
\[
\LRp{\bm{Q}^i}^T \bm{1} = \bm{E}^T\bm{B}^i \bm{1}.
\]
However, under Assumption~\ref{ass:quad}, the entries of $\LRp{\bm{Q}^i}^T \bm{1}$ are exactly $\int_{\hat{D}} \pd{\phi_j}{\hat{\bm{x}}_i}$ and the entries of $\bm{E}^T\bm{B}^i \bm{1}$ are exactly $\int_{\partial \hat{D}} \phi_j(\bm{x}) \hat{n}_i$ since $\phi_j(\bm{x}) \in V^N$.  These two terms are then identical by the exactness of integrals and fundamental theorem of calculus. }

\qed\end{proof}
In Sections~\ref{sec:singleelem} and \ref{sec:curved}, specific polynomials $v(\bm{x})$ will be motivated by \note{the extension of the proof of entropy stability on curved elements}, and we will present examples of volume and surface quadrature rules on simplicial and tensor product elements which satisfy Assumption~\ref{ass:quad} for these choices of $v$.  

\subsection{On quadrature conditions for Assumption~\ref{ass:quad} with $v = 1$}
\label{sec:assump1}
Apart from algebraic manipulations, only Lemma~\ref{lemma:sbpcor} is necessary to prove entropy conservation in Theorem~\ref{thm:esdg}.  Lemma~\ref{lemma:sbpcor} requires that Assumption~\ref{ass:quad} holds for $v=1$.  Thus, the volume and surface quadratures must be sufficiently accurate to guarantee that the mass matrix is positive definite and to integrate 
\begin{equation}
\int_{\hat{D}} \pd{u}{x_i}, \qquad \int_{\partial \hat{D}} u \hat{n}_i. \label{eq:affineints}
\end{equation}
On simplicial elements, the mass matrix is guaranteed to be positive definite for any volume quadrature which is exact for degree $2N$ polynomial integrands.  This choice of volume quadrature also guarantees that the volume term in (\ref{eq:affineints}) is integrated exactly.  The surface quadrature can thus be taken to be any quadrature rule which is exact for only degree $N$ integrands on faces.  In contrast, the construction of simplicial decoupled SBP operators has required face quadratures which are accurate for degree $2N$ polynomials \cite{chan2017discretely, chan2018discretely}.  

On tensor product elements, we can take any degree $(2N-1)$ quadrature rule which ensures a positive definite mass matrix (e.g.\ a $(N+1)$-point GLL quadrature), as a quadrature of this accuracy is sufficient to exactly integrate the volume term in (\ref{eq:affineints}).  For the surface quadrature, we can again take any quadrature rule which is exact for degree $N$ polynomial integrands.  For example, on a quadrilateral element, one can use $\left\lceil\frac{N+1}{2}\right\rceil$-point Gauss quadrature rule or a $\left\lceil\frac{N+3}{2}\right\rceil$-point GLL rule as face quadratures for a degree $N$ scheme.

On tensor product elements, we restrict ourselves to isotropic volume quadrature rules which are construced from tensor products of one-dimensional quadrature formulas.  For the remainder of this work, the degree of the multi-dimensional quadrature rule on tensor product elements will refer to the degree of exactness of the one-dimensional rule.  For example, we refer to the quadrature rule constructed through a tensor product of one-dimensional $(N+1)$-point GLL quadrature rules as a degree $(2N-1)$ quadrature rule.  This choice of quadrature is sufficient to guarantee that the mass matrix is positive definite \cite{canuto2007spectral}.

\subsection{On the accuracy of \bnote{skew-hybridized SBP operators }}
\label{sec:accskew}

It was shown in \cite{chan2017discretely, chan2018efficient} that the \bnote{hybridized SBP operator $\bm{Q}^i_N$} can be interpreted as augmenting a volume approximation of the derivative with boundary correction terms.  Let $f(\bm{x}),g(\bm{x})$ denote two $L^2$ integrable functions, and let $\bm{f}_N, \bm{g}_N$ denote the vectors of values of $f,g$ at both volume and surface quadrature points.  A degree $N$ approximation $u\in V^N$ to $f\pd{g}{x_i}$ can be constructed via
\begin{equation}
\bm{M}\bm{u} = \LRs{\begin{array}{c}
\bm{V}_q \\ \bm{V}_f\end{array}}^T \diag{\bm{f}_N}\bm{Q}^i_N\bm{g}_N, \qquad \bm{f}_N = \begin{bmatrix} \bm{f}_q\\ \bm{f}_f\end{bmatrix}, \quad \bm{g}_N = \begin{bmatrix} \bm{g}_q\\ \bm{g}_f\end{bmatrix}, 
\label{eq:dsbpmatrixform}
\end{equation}
where $\bm{u}$ denotes the vector of coefficients for $u$.  

This algebraic expression (\ref{eq:dsbpmatrixform}) can be reinterpreted as a quadrature approximation of a variational problem, which can be mapped to a physical element $D^k$.  We seek to approximate $f\pd{g}{x_i}$ by $u \in V^N$ such that, $\forall v\in V^N$
\begin{equation}
\int_{D^k} u v = \int_{D^k} f\pd{\Pi_N g}{{x}_i}v + \int_{\partial {D}^k} \LRp{g - \Pi_Ng} \LRp{\frac{fv + \Pi_N(fv)}{2}} {n}^k_i,\label{eq:var}
\end{equation}
where $\Pi_N$ is the $L^2$ projection operator (\ref{eq:l2proj}).  
Integrating half of the volume term by parts yields the skew-symmetric form of (\ref{eq:var})
\begin{align}
\int_{D^k} u  v &=\frac{1}{2} \int_{{D^k}} \LRp{f\pd{\Pi_N g}{{x}_i}v - g \pd{\Pi_N\LRp{fv}}{{x}_i}} \label{eq:skewDN}\\
&+\frac{1}{2} \int_{\partial {D}^k} \LRp{fgv + \LRp{g - \Pi_Ng} \LRp{{fv + \Pi_N(fv)}{}}} {n}^k_i \qquad \forall v\in V^N\nonumber,
\end{align}
which yields a matrix formulation \bnote{involving the skew-hybridized SBP operator $\tilde{\bm{Q}}^i_N$}
\begin{align}
\bm{M}\bm{u} =& \frac{1}{2}\LRs{\begin{array}{c}
\bm{V}_q \\ \bm{V}_f\end{array}}^T \diag{\bm{f}_N}\underbrace{\LRp{\bm{Q}^i_N - \LRp{\bm{Q}^i_N}^T + \bm{B}^i_N}}_{\tilde{\bm{Q}}^i_N}\bm{g}_N \label{eq:dsbpmatrixform2}.
\end{align}

The accuracy of the formulation (\ref{eq:esdgSkew}) can be understood by analyzing the degree of polynomial exactness of (\ref{eq:dsbpmatrixform2}) as an approximation of the derivative.  Let $u(\bm{x})$ be a polynomial of degree $\leq N$ with coefficients $\bm{u}$, and let $\bm{u}_N = \LRs{\bm{u}_q, \bm{u}_f}^T$ denote the values of $u(\bm{x})$ at volume and surface quadrature points.  An approximation of $\pd{u}{x_i}$ can be computed by applying (\ref{eq:dsbpmatrixform}) to compute
\begin{align}
\pd{u}{x_i}\approx 
\bm{M}^{-1}\LRs{\begin{array}{c}
\bm{V}_q \\ \bm{V}_f\end{array}}^T \tilde{\bm{Q}}^i_N\bm{u}_N.
\label{eq:dsbpapprox}
\end{align}
From (\ref{eq:dsbpapprox}), it can be shown that when $\bm{Q}^i$ satisfies a generalized SBP property, the hybridized SBP operator (\ref{eq:dsbp}) produces a degree $N$ approximation to the derivative \cite{chan2017discretely}.  When $\bm{Q}^i$ does not satisfy a generalized SBP property, we have the following lemma on the accuracy of (\ref{eq:dsbpapprox}):
\begin{lemma}
Let $M\leq N$.  Suppose that the volume quadrature is exact for degree $M+N-1$ polynomials on simplices, or for polynomials in $Q^{M+N-1,M+N,M+N}$ on tensor product elements.  Furthermore, assume that the surface quadrature is exact for degree $M+N$ polynomials on simplices and $Q^{M+N,M+N}$ on tensor product elements.  \bnote{Then, so long as the mass matrix is positive definite}, the skew-symmetric approximation of the $x$-derivative (\ref{eq:dsbpapprox}) is exact for polynomials of degree $M$.  
\label{lemma:dsbpapprox}
\end{lemma}
\begin{proof}
\bnote{Suppose $u \in V^M$.  Let $\bm{u}$ denote the polynomial coefficients of $u$, and let $\bm{e}$ denote the difference between $\bm{D}^i\bm{u}$ (the exact coefficients of $\pd{u}{x_i}$) and the approximation (\ref{eq:dsbpapprox})
\[
\bm{e} = \bm{D}^i\bm{u} - \bm{M}^{-1}\LRs{\begin{array}{c}
\bm{V}_q \\ \bm{V}_f\end{array}}^T\tilde{\bm{Q}}^i_N\bm{u}_N
\]
where $\bm{e}$ is a polynomial of degree $N$.  
Since $u(\bm{x})$ is polynomial, the values of $u(\bm{x})$ at quadrature points are $\bm{u}_q = \bm{V}_q\bm{u}$ and $\bm{u} = \bm{P}_q\bm{u}_q$.  This implies that $\bm{u}_f = \bm{V}_f \bm{P}_q\bm{u}_q =  \bm{E}\bm{u}_q$.  Expanding the latter term yields 
\begin{align*}
\bm{M}^{-1}\LRs{\begin{array}{c}
\bm{V}_q \\ \bm{V}_f\end{array}}^T\tilde{\bm{Q}}^i_N\bm{u}_N &= \frac{1}{2}\bm{M}^{-1}\LRs{\begin{array}{c}
\bm{V}_q \\ \bm{V}_f\end{array}}^T\begin{bmatrix}
\bm{Q}^i - \LRp{\bm{Q}^i}^T & \bm{E}^T\bm{B}^i\\
-\bm{B}^i\bm{E} & \bm{B}^i
\end{bmatrix}\begin{bmatrix}
\bm{u}_q\\
\bm{u}_f
\end{bmatrix}\\
&= \frac{1}{2}\bm{M}^{-1}\LRs{\begin{array}{c}
\bm{V}_q \\ \bm{V}_f\end{array}}^T\begin{bmatrix}
\bm{Q}^i\bm{u}_q + \LRp{\bm{E}^T\bm{B}^i\bm{E} - \LRp{\bm{Q}^i}^T}\bm{u}_q\\
\bm{B}^i\LRp{\bm{u}_f - \bm{E}\bm{u}_q}
\end{bmatrix} 
\\
&=\frac{1}{2}\bm{M}^{-1}
\bm{V}_q^T\LRp{\bm{Q}^i\bm{u}_q + \LRp{\bm{E}^T\bm{B}^i\bm{E} - \LRp{\bm{Q}^i}^T}\bm{u}_q}.
\end{align*}}
Since $\bm{Q}^i = \bm{W}\bm{V}_q\bm{D}^i\bm{P}_q$ and $\bm{M} = \bm{V}_q^T\bm{W}\bm{V}_q$, we have that $\bm{M}^{-1}\bm{V}_q^T\bm{Q}^i\bm{u}_q = \bm{D}^i\bm{u}$.  This simplifies the expression for error to
\begin{align}
\bm{e}^T\bm{M}\bm{e} &= \frac{1}{2}
\bm{e}_q^T\LRp{-\bm{Q}^i\bm{u}_q + \LRp{\bm{E}^T\bm{B}^i\bm{E} - \LRp{\bm{Q}^i}^T}\bm{u}_q} \label{eq:skewerr}\\
&= -\frac{1}{2}
\bm{e}_q^T\bm{Q}^i\bm{u}_q + \frac{1}{2}\bm{e}_q^T\LRp{\bm{E}^T\bm{B}^i\bm{E} - \LRp{\bm{Q}^i}^T}\bm{u}_q, \nonumber
\end{align}
where we have introduced $\bm{e}_q = \bm{V}_q\bm{e}$.  
Since $u \in V^M$ and $e \in V^N$, by exactness of the quadrature rules, 
\[
\bm{e}_q^T\LRp{\bm{E}^T\bm{B}^i\bm{E} - \LRp{\bm{Q}^i}^T}\bm{u}_q = \int_{\partial \hat{D}} u e n_i - \int_{\hat{D}} u \pd{e}{x_i}
= \int_{\hat{D}} \pd{u}{x_i} e = \bm{e}_q^T\bm{Q}^i\bm{u}_q.
\]
Combining this with (\ref{eq:skewerr}) implies that $\bm{e}^T\bm{M}\bm{e} = 0$, and since $\bm{M}$ is symmetric positive definite, $\bm{e} = 0$.  
\qed\end{proof}


Lemma~\ref{lemma:dsbpapprox} suggests that, when a generalized SBP property does not hold, the use of under-integrated quadratures results in a loss of one or more orders of accuracy.  For example, if the SBP property does not hold, then using $(N+1)$ point GLL rules (which are exact for only polynomials of degree $2N-1$) for either volume or surface quadratures should result in a loss of one order of accuracy compared to the use of $(N+1)$-point Gauss rules (which are exact for polynomials of degree $2N$).  This is indeed observed in numerical experiments.

\section{Skew-symmetric entropy conservative formulations on mapped elements}
\label{sec:skew2}

We now construct skew-symmetric formulations on mapped elements.  We assume some domain $\Omega$ is decomposed into non-overlapping elements $D^k$, such that $D^k$ is the image of the reference element $\hat{D}$ under an isoparametric mapping $\bm{\Phi}^k$.  We define geometric change of variables terms ${G}^k_{ij}$ as scaled derivatives of reference coordinates $\hat{\bm{x}}$ w.r.t.\ physical coordinates $\bm{x}$
\begin{gather}
\pd{u}{x_i} = \sum_{j=1}^d {G}^k_{ij}\pd{u}{\hat{x}_j}, \qquad {G}^k_{ij} = J^k\pd{\hat{x}_j}{{x}_i}, 
\label{eq:geofacs}
\end{gather}
where $J^k$ is the determinant of the Jacobian of the geometric mapping on the element $D^k$.  We also introduce the scaled outward normal components $n_iJ^k_f$, which can be computed in terms of (\ref{eq:geofacs}) and the reference normals $\hat{\bm{n}}$ on $\hat{D}$
\begin{gather}
n^k_i J^k_f = \sum_{j=1}^d G^k_{ij} \hat{{n}}_j.  
\label{eq:normals}
\end{gather}
We also define $\bm{n}^k_i$ as the vector containing concatenated values of the scaled outward normals $n^k_iJ^k_f$ at surface quadrature nodes.  For the remainder of the work, we assume that the mesh is watertight or \note{``well-constructed'' \cite{kopriva2016geometry, chan2018discretely, kopriva2019free}} such that at all points on any internal face, the scaled outward normals $n^k_iJ^k_f$ on the two elements sharing this face are equal and opposite.  

As shown in the previous section, on a single element (and on affine meshes), it is possible to guarantee entropy stability of the skew-symmetric formulation (\ref{eq:esdgSkew}) under a surface quadrature which is only exact for degree $N$ polynomials.  However, on curved meshes, stronger conditions are required to guarantee entropy stability.  This is due to the fact that the geometric terms are now high order polynomials which vary spatially over each element.  Moreover, Lemma~\ref{lemma:sbpcor} assumes affine geometric mappings, and does not hold on curved elements.  In this section, we discuss how to extend Lemma~\ref{lemma:sbpcor} to curved simplicial and tensor product elements.  

\subsection{Curved elements and the geometric conservation law}
\label{sec:curved}

In this section, we describe how to construct appropriate hybridized SBP operators on curved meshes, and give conditions on the volume and surface quadrature rules under which a semi-discretely entropy stable scheme can be constructed.


We first show how to construct appropriate SBP operators on curved elements.  Let $\bm{G}^k_{ij}$ denote the vector of scaled geometric terms ${G}^k_{ij}$ evaluated at both volume and surface quadrature points, and let $\tilde{\bm{Q}}^j_N$ now denote the skew-symmetric construction of the hybridized SBP operator for the $j$th reference coordinate.  Hybridized SBP operators on a curved element $D^k$ can be defined as in \cite{chan2018discretely} by
\begin{equation}
\bm{Q}^i_k = \frac{1}{2}\sum_{j=1}^d \LRp{\diag{\bm{G}^k_{ij}}\tilde{\bm{Q}}^j_N + \tilde{\bm{Q}}^j_N\diag{\bm{G}^k_{ij}}}.
\label{eq:dncurved}
\end{equation}
\bnote{Since $\tilde{\bm{Q}}^j_N$ satisfies a summation by parts property on the reference element $\hat{D}$, then $\bm{Q}^i_k$ satisfies an analogous SBP property on the physical element $D^k$ \cite{chan2018discretely}. }

We can now construct and prove entropy conservation and free stream preservation for a skew-symmetric formulation on a physical curved element $D^k$.  \bnote{Free stream preservation is necessary to discretely preserve both entropy conservation and the condition that constant solutions are stationary solutions of systems of conservation laws.  However, on curved meshes, the presence of spatially varying geometric terms can result in the production of spurious transient waves.  The construction of geometric terms through (\ref{eq:iconscurl}) guarantees that the resulting methods are free-stream preserving, and that constant solutions remain stationary solutions of discretizations of (\ref{eq:nonlineqs}).  }

Let $\bm{Q}^i_k$ be given by (\ref{eq:dncurved}), and define the curved mass matrix 
\[
\bm{M}^k = \bm{V}_q^T\bm{W}\diag{\bm{J}^k}\bm{V}_q.
\]
Note that $\bm{M}^k$ is positive-definite so long as $J^k$ is positive at all quadrature points.  We define the auxiliary quantities $\tilde{\bm{u}}$ 
\begin{gather*}
\bm{v}_q = \bm{v}\LRp{\bm{u}_q}, \qquad \tilde{\bm{v}} = \begin{bmatrix}
\bm{V}_q\\
\bm{V}_f
\end{bmatrix}\bm{P}^k_q\bm{v}_q, \qquad \tilde{\bm{u}} = \bm{u}\LRp{\tilde{\bm{v}}}.
\end{gather*}
where $\bm{P}^k_q = \LRp{\bm{M}^k}^{-1}\bm{V}_q^T\bm{W}\diag{\bm{J}^k}$.  Then, we have the following theorem:
\begin{theorem}
\label{thm:skewformcurved}
\bnote{Assume that $\bm{Q}^i_k\bm{1}=\bm{0}$.}
Let $\tilde{\bm{u}}_f^+$ denote the face value of the entropy-projected conservative variables $\tilde{\bm{u}}_f$ on the neighboring element.  Then, the formulation
\begin{gather}
\bm{M}^k\pd{\bm{u}_h}{t} + 
\sum_{i=1}^d \LRs{\begin{array}{cc}
\bm{V}_q \\
\bm{V}_f\end{array}}^T 2\LRp{\bm{Q}^i_k \circ \bm{F}^i_S}\bm{1} + \bm{V}_f^T \bm{W}_f \diag{\bm{n}^k_i}\LRp{\bm{f}_i^*-\bm{f}(\tilde{\bm{u}}_f)} = 0,
\label{eq:skewformcurved}\\
\LRp{\bm{F}^i_S}_{ij} = \bm{f}^i_S\LRp{\tilde{\bm{u}}_i,\tilde{\bm{u}}_j}, \qquad 1\leq i,j\leq N_q + N^f_q, \nonumber\\
\bm{f}_i^* = \bm{f}^i_S(\tilde{\bm{u}}_f^+,\tilde{\bm{u}}_f), \qquad \text{ on interior interfaces,} \nonumber
\end{gather}
is semi-discretely entropy conservative on $D^k$ such that for $\bm{u}_q = \bm{V}_q\bm{u}$,
\begin{gather*}
\bm{1}^T\bm{W}\diag{\bm{J}^k}\pd{U(\bm{u}_q)}{t} + \sum_{i=1}^d\bm{1}^T\bm{W}_f \diag{\bm{n}^k_i} \LRp{\psi_i(\tilde{\bm{u}}_f) - \tilde{\bm{v}}_f^T\bm{f}_i^*} = 0.
\end{gather*}
Additionally, the method is free-stream preserving such that $\pd{\bm{u}_h}{t} = 0$ for constant solutions.
\label{thm:esdgCurved}
\end{theorem}
\bnote{We omit the proof of entropy conservation, since it is identical to the proofs in \cite{chan2017discretely, chan2018discretely}.  Free-stream preservation follows directly from $\bm{Q}^i_k\bm{1} = \bm{0}$ and the fact that $\bm{F}_S$ is constant for constant solutions \cite{kopriva2006metric}. }


\bnote{The proof of Theorem~\ref{thm:esdgCurved} requires $\bm{Q}^i_k\bm{1} = \bm{0}$.  For curved elements, additional steps must also be taken to ensure this condition.  Assuming Assumption~\ref{ass:quad} holds for $v(\bm{x}) = 1$ and expanding out the expression for $\bm{Q}^i_k\bm{1} = \bm{0}$ using (\ref{eq:dncurved}) yields}
\begin{align}
\bm{Q}^i_k \bm{1} = \frac{1}{2} \sum_{j=1}^d \diag{\bm{G}^k_{ij}}\tilde{\bm{Q}}^j_N \bm{1} + \tilde{\bm{Q}}^j_N\diag{\bm{G}^k_{ij}}\bm{1} = \frac{1}{2}\sum_{j=1}^d \tilde{\bm{Q}}^j_N\LRp{\bm{G}^k_{ij}} = 0,
\label{eq:dgcl}
\end{align}
\bnote{where we have used that $\tilde{\bm{Q}}^j_N \bm{1} = 0$ using Lemma~\ref{lemma:sbpcor}.}  We refer to the condition $\bm{Q}^i_k\bm{1} = 0$ as the discrete geometric conservation law (GCL) \cite{thomas1979geometric, kopriva2006metric}.  For degree $N$ isoparametric mappings, the GCL is automatically satisfied in two dimensions due to the fact that the exact geometric terms ${G}^k_{ij}$ are polynomials of degree $N$ \cite{kopriva2006metric}.  However, in three dimensions, the GCL is not automatically satisfied due to the fact that the degree of $G^k_{ij}$ is larger than $N$.  Thus, the geometric terms cannot be represented exactly using degree $N$ polynomials, and (\ref{eq:dgcl}) must be enforced through an alternative construction of ${G}^k_{ij}$.

To ensure that the geometric terms satisfy the GCL, we first rewrite the geometric terms as the curl of some quantity $\bm{r}^i$, but interpolate $\bm{r}^i$ before applying the curl \cite{thomas1979geometric, visbal2002use, kopriva2006metric, hindenlang2012explicit, chan2018discretely}:
\begin{gather}
\bm{r}^i = { \pd{\bm{x}}{\hat{x}_i}\times \bm{x}}, \qquad
\LRs{\begin{array}{c}
{G}^k_{1j}\\
{G}^k_{2j}\\
{G}^k_{3j}\end{array}} = \begin{bmatrix}
\LRp{-\hat{\Grad}\times I_{N_{\rm geo}}\LRp{x_3\hat{\Grad}x_2}}_j\\
\LRp{\hat{\Grad}\times I_{N_{\rm geo}}\LRp{x_3\hat{\Grad}x_1}}_j\\
\LRp{\hat{\Grad}\times I_{N_{\rm geo}}\LRp{x_1\hat{\Grad}x_2}}_j
\end{bmatrix},\label{eq:iconscurl} \\
N_{\rm geo} \leq \begin{cases}
N+1 & \text{(tetrahedra)}\\
N & \text{(hexahedra)}
\end{cases},\nonumber
\end{gather}
where $I_{N_{\rm geo}}$ denotes a degree $N_{\rm geo}$ polynomial interpolation operator with appropriate interpolation nodes.\footnote{This interpolation step must be performed using interpolation points with an appropriate number of nodes on each boundary \cite{chan2018discretely}.  These include, for example, GLL nodes on tensor product elements, and optimized interpolation nodes on non-tensor product elements \cite{hesthaven1998electrostatics, warburton2006explicit, chan2015comparison}.}  The restriction on the maximum value of $N_{\rm geo}$ ensures that $G^k_{ij} \in V^N$ (e.g. $G^k_{ij} \in P^N$ on tetrahedral elements and $G^k_{ij}\in Q^N$ on hexahedral elements), which is also necessary to guarantee (\ref{eq:dgcl}).

Because the \bnote{skew-hybridized SBP operators $\bm{Q}^i_k$} are now defined through (\ref{eq:dncurved}), Lemma~\ref{lemma:sbpcor} and the proof of entropy stability no longer hold for curved elements and must be modified.  The introduction of curvilinear meshes will impose slightly different conditions on the accuracy of the surface quadrature.  We discuss simplicial and tensor product elements separately, as differences in the natural polynomial approximation spaces will result in different assumptions for each proof.

\begin{lemma}
\label{lemma:vdsbpcurved} 
Let $D^k$ be a curved element, and let the geometric terms $G^k_{ij}$ be constructed using \rnote{(\ref{eq:iconscurl})}.  Let Assumption~\ref{ass:quad} hold for $v = 1$ and $v = G^k_{ij}$ for all $i,j = 1,\ldots, d$.  Then,
\[
\qquad \bm{Q}^i_k\bm{1} = \bm{0},
\]
\end{lemma}
\begin{proof}
The proof of $\bm{Q}^i_k\bm{1} = \bm{0}$ is analogous to the proof of Lemma~\ref{lemma:sbpcor}.  The results follow for tensor product elements using results from \cite{kopriva2006metric} and for simplicial elements using results from \cite{chan2018discretely}.   In both cases, the proof relies only on the fact that $G^k_{ij} \in V^N$.  
\qed\end{proof}

The proof of global entropy conservation follows from summing up (\ref{eq:esdgthm}) over all elements and noting that the surface terms cancel due to the symmetry and conservation properties of the Tadmor flux (\ref{eq:esflux}) \cite{chan2017discretely}.  The entropy conservative formulations presented in this work can be made entropy stable by adding appropriate interface dissipation, such as Lax-Friedrichs or matrix-based penalization terms \cite{winters2017uniquely, chen2017entropy, chan2017discretely}.  


\begin{remark}
It is also possible to replace the curved mass matrix $\bm{M}^k$ with a more easily invertible weight-adjusted approximation while maintaining high order accuracy, entropy stability, and local conservation \cite{chan2018discretely}.  This approximation avoids the inversion of dense weighted $L^2$ mass matrices $\bm{M}^k$ on curved simplicial elements, but is generally unnecessary on tensor product elements as common choices of volume quadrature result in a diagonal (lumped) mass matrix \cite{carpenter2014entropy, parsani2016entropy, chan2018efficient}.
\end{remark}

\subsection{On quadrature conditions for Assumption~\ref{ass:quad} for $v = 1$ and $v = G^k_{ij}$}
\label{sec:quadacc}

The previous sections outline minimal conditions under which entropy stability is guaranteed under a skew-symmetric formulation and a polynomial geometric mapping.  In this section, we translate these minimial conditions into conditions on quadrature accuracy.  

Semi-discrete entropy conservation on curved meshes requires that Assumption~\ref{ass:quad} holds for $v = 1$ and $v = G^k_{ij}$.  We discuss specific choices of volume and surface quadrature for which this assumption is valid, and summarize the maximum degree $N_{\rm geo}$ of the polynomial geometric approximation under which entropy stability holds for common choices of volume and surface quadrature.  

In order to ensure that the mass matrix is positive-definite in Assumption~\ref{ass:quad}, the volume quadrature must be degree $2N$ in general on simplices.  The following lemma summarizes expected behavior for surface quadrature rules of varying order:
\begin{lemma}
Let $\hat{D}$ be a simplex with volume quadrature which is exact for degree $2N$ polynomials.  Let the surface quadrature be exact for polynomials of degree $M+N$.  Then, the skew-symmetric formulation (\ref{eq:skewformcurved}) is entropy stable for $N_{\rm geo} \leq \min\LRp{N+1,M+1}$.
\label{lemma:curvsimp}
\end{lemma}
\begin{proof}
Entropy stability holds if Assumption~\ref{ass:quad} holds for $v = 1$ and $v = G^k_{ij}$.  
Simplicial elements require $N_{\rm geo}\leq (N+1)$ in order to guarantee that $G^k_{ij}\in P^{N_{\rm geo}-1} \subset P^{N}$, which is necessary to satisfy the discrete GCL \cite{chan2018discretely}.  Then, for $u \in P^N$, $\pd{u}{\hat{x}_j}\in P^{N-1}$, the integrands in Assumption~\ref{ass:quad} are $\pd{u}{\hat{x}_j} v \in P^{N+N_{\rm geo}-2}$ and $uv n_i \in P^{N+N_{\rm geo}-1}$ for $v = G^k_{ij}$.  The volume quadrature exactly integrates the first integrand for $N_{\rm geo} \leq N+1$, while the surface quadrature exactly integrates the second integrand for $M \geq N_{\rm geo}-1$, or $N_{\rm geo} \leq M+1$.
\qed\end{proof}

The situation is more complicated for curved tensor product elements.  It was shown in Section~\ref{sec:assump1} that tensor product quadratures of degree $(2N-1)$ satisfy Assumption~\ref{ass:quad} for $v = 1$.  However, in contrast to the simplicial case, it is not immediately clear that degree $(2N-1)$ volume quadratures exactly integrate $\int_{\hat{D}} \pd{u}{\hat{x}_j}v$ for $v = G^k_{ij}$ for tensor product elements.  The difference between simplicial and tensor product elements is the polynomial space in which the derivative lies.  In contrast to the simplicial case, if $u \in Q^N$, $\pd{u}{\hat{x}_j} \not\in Q^{N-1}$.  Consider the three-dimensional case with $u, v \in Q^N$ and $i = 1$.  Then, differentiation reduces the polynomial degree in one coordinate but not others and $\pd{u}{\hat{x}_1} \in Q^{N-1,N,N}$.  As a result, $\pd{u}{\hat{x}_j}v \not\in Q^{2N-1}$, and a tensor product quadrature of degree $(2N-1)$ (in each coordinate) does not exactly integrate $\int_{\hat{D}} \pd{u}{\hat{x}_j}v$ for general $v\in Q^N$.  

We address the quadrilateral case first:
\begin{lemma}
Let $\hat{D}$ be a quadrilateral.  Suppose the volume quadrature be exact for degree $M+N$ polynomials, and that the surface quadrature be exact for polynomials of degree $M+N$.  Then, the skew-symmetric formulation (\ref{eq:skewformcurved}) is entropy stable for $N_{\rm geo} \leq \min\LRp{N,M+1}$.
\label{lemma:curvquad}
\end{lemma}
\begin{proof}
As in Lemma~\ref{lemma:curvsimp}, entropy stability holds if Assumption~\ref{ass:quad} holds for $v = 1$ and $v = G^k_{ij}$.  The case of $v = 1$ was addressed previously, and we focus on $v = G^k_{ij}$.  We first characterize the polynomial degree of the geometric terms $G^k_{ij}$.  In contrast to the simplicial case, tensor product elements require $N_{\rm geo}\leq N$ in order to ensure that $G^k_{ij}\in Q^{N,N}$ and that the discrete GCL is satisfied \cite{kopriva2006metric}.  On a quadrilateral element with a degree $N_{\rm geo}$ geometric mapping, $G^k_{ij}$ is
\begin{align*}
G^k_{11} &= \pd{x_2}{\hat{x}_2} \in Q^{N_{\rm geo}, N_{\rm geo}-1}, \qquad G^k_{12} = -\pd{x_2}{\hat{x}_1} \in Q^{N_{\rm geo}-1, N_{\rm geo}},\\
G^k_{21} &= -\pd{x_1}{\hat{x}_2} \in Q^{N_{\rm geo}, N_{\rm geo}-1}, \qquad G^k_{22} = \pd{x_1}{\hat{x}_1}\in Q^{N_{\rm geo}-1, N_{\rm geo}}.\nonumber
\end{align*}
Since $\pd{u}{\hat{x}_1} \in Q^{N-1,N}$ and $\pd{u}{\hat{x}_2} \in Q^{N,N-1}$
\[
\pd{u}{\hat{x}_i}G^k_{ij} \in Q^{N+N_{\rm geo}-1}.
\]
The volume quadrature exactly integrates this integrand for $M \geq N_{\rm geo}-1$.

We now consider the condition in Assumption~\ref{ass:quad} on the surface integrals $\int_{\partial \hat{D}} u v \hat{n}_j$ for $v= G^k_{ij}$.  For left and right faces of the quadrilateral, $\hat{n}_2 = 0$, so this condition reduces to ensuring that the quantity $u G^k_{i1}$ is integrated exactly using quadrature for $i = 1,2$.  Since $G^k_{i1}$ are degree $N_{\rm geo}-1$ in the $\hat{x}_2$ coordinate, $G^k_{i1}$ is degree $N_{\rm geo}-1$ and $uG^k_{i1} \hat{n}_1 \in Q^{N+N_{\rm geo}-1}$ along the left and right faces.  Similarly, $uG^k_{i2} \hat{n}_2 \in Q^{N+N_{\rm geo}-1}$ along the top and bottom faces and are zero along the left and right faces.  The surface quadrature rule exactly integrates such integrands for $M \geq N_{\rm geo} - 1$, or $N_{\rm geo} \leq M+1$.
\qed\end{proof}

Existing proofs of entropy stability on quadrilaterals rely on $(N+1)$-point GLL volume and surface quadratures, which are exact for degree $2N-1$ polynomials.   The novelty of Lemma~\ref{lemma:curvquad} is that the proof holds for any combination of degree $2N-1$ volume and surface quadratures (for example, $(N+1)$-point GLL volume quadrature and an $(N-1)$-point Gauss surface quadrature).  

We now consider the three-dimensional case.  In contrast to the quadrilateral case, the GCL is not automatically satisfied for a degree $N_{\rm geo} \leq N$ geometric mapping.  Instead, GCL-satisfying geometric terms are approximated using (\ref{eq:iconscurl}).  Expanding out the expression for $G^k_{11}$ gives
\[
G^k_{11} = \pd{}{\hat{x}_3} I_{N_{\rm geo}}\LRp{{x}_3 \pd{x_2}{\hat{x}_2}} - \pd{}{\hat{x}_2} I_{N_{\rm geo}}\LRp{{x}_3 \pd{x_2}{\hat{x}_3}} \in Q^{N_{\rm geo}}.
\]
Repeating for the other geometric terms, one can show that $G^k_{ij} \in Q^{N_{\rm geo}}$ on hexahedral elements.  Thus, if $u\in Q^N$, $\pd{u}{\hat{x}_1} G^k_{i1} \in Q^{N+N_{\rm geo}-1,N+N_{\rm geo},N+N_{\rm geo}}$, and is only integrated exactly by volume quadratures of degree $(2N-1)$ for geometric degrees $N_{\rm geo}\leq (N-1)$.  Similarly, Assumption~\ref{ass:quad} does not hold under degree $(2N-1)$ surface quadratures unless $N_{\rm geo} \leq (N-1)$, due to the fact that traces of $G^k_{ij}$ are degree $N_{\rm geo}$ polynomials in each coordinate.\footnote{It is possible to  construct the geometric terms for $N_{\rm geo} = N$ using a local $H_{\rm div}$ basis where 
\[
\bm{r}^i \in Q^{N-1,N,N} \times Q^{N,N-1,N} \times Q^{N,N,N-1}.
\]  
Then, the geometric terms satisfy $\Grad \times \bm{r}^i \in Q^{N,N-1,N-1}\times Q^{N-1,N,N-1} \times Q^{N-1,N-1,N}$ with traces in $Q^{N-1}$, and Assumption~\ref{ass:quad} holds under degree $(2N-1)$ volume and surface quadrature.  This approach will be investigated in more detail in future work.  
}
We summarize these findings in the following lemma for hexahedral elements:
\begin{lemma}
Let $\hat{D}$ be a hexahedral element, with geometric terms constructed using (\ref{eq:iconscurl}).  Let the volume quadrature be exact for degree $M+N$ polynomials, and let the surface quadrature be exact for polynomials of degree $M+N$.  Then, the skew-symmetric formulation (\ref{eq:skewformcurved}) is entropy stable for $N_{\rm geo} \leq \min\LRp{N,M}$.
\label{lemma:curvquad}
\end{lemma}

Most implementations on tensor product elements utilize volume and surface quadratures of either degree $(2N-1)$ or $2N$.  We summarize below for different pairings of volume and surface quadrature the maximum degree $N_{\rm geo}$ under which Assumption~\ref{ass:quad} is satisfied and entropy stability is guaranteed:
\begin{enumerate}
\item On quadrilateral elements, Assumption~\ref{ass:quad} holds for $N_{\rm geo} \leq N$ and any tensor product volume and surface quadratures of degree $(2N-1)$ 
\item On hexahedral elements, Assumption~\ref{ass:quad} holds for $N_{\rm geo} \leq N-1$ and any tensor product volume and surface quadratures of degree $(2N-1)$.  If the SBP property holds (e.g.\ for GLL quadrature, or for volume and surface quadratures of degree $2N$) then Assumption~\ref{ass:quad} holds for $N_{\rm geo} \leq N$.
\end{enumerate}

We note that the condition $N_{\rm geo} \leq N-1$ is non-standard for tensor product elements.  However, this condition is only necessary for entropy stability when $\bm{Q}^i$ does not satisfy a generalized SBP property (see Remark~\ref{remark:dsbp}).  To the author's knowledge, this setting has not been considered within the literature.

\section{Numerical experiments}
\label{sec:num}

In this section, we present two-dimensional experiments which verify the theoretical results presented and qualify the accuracy of the proposed methods.  We begin by investigating the maximum stable timestep, stability, and accuracy of the skew-symmetric formulation on triangular and quadrilateral meshes, and conclude with two-dimensional experiments on a hybrid mesh containing mixed quadrilateral and triangular elements.  

We consider numerical solutions of the 2D compressible Euler equations
\begin{align*}
\pd{\rho}{t} + \pd{\LRp{\rho u}}{x_1} + \pd{\LRp{\rho v}}{x_2} &= 0,\\
\pd{\rho u}{t} + \pd{\LRp{\rho u^2 + p }}{x_1} + \pd{\LRp{\rho uv}}{x_2} &= 0,\nonumber\\
\pd{\rho v}{t} + \pd{\LRp{\rho uv}}{x_1} + \pd{\LRp{\rho v^2 + p }}{x_2} &= 0,\nonumber\\
\pd{E}{t} + \pd{\LRp{u(E+p)}}{x_1} + \pd{\LRp{v(E+p)}}{x_2}&= 0,\nonumber
\end{align*}
where we have introduced the pressure is $p = (\gamma-1)\LRp{E - \frac{1}{2}\rho (u^2+v^2)}$ and the specific internal energy $\rho e = E - \frac{1}{2}\rho (u^2+v^2)$.  We seek entropy stability with respect to the entropy for the compressible Navier-Stokes equations \cite{hughes1986new}
\begin{equation*}
U(\bm{u}) = -\frac{\rho s}{\gamma-1},
\label{eq:entropy2d}
\end{equation*}
where $s = \log\LRp{\frac{p}{\rho^\gamma}}$ denotes the specific entropy. The mappings between conservative and entropy variables in two dimensions are given by
\begin{align*}
v_1 &= \frac{\rho e (\gamma + 1 - s) - E}{\rho e}, \qquad v_2 = \frac{\rho u}{\rho e}, \qquad v_3 = \frac{\rho v}{\rho e}, \qquad v_4 = -\frac{\rho}{\rho e}\\
\rho &= -(\rho e) v_4, \qquad \rho u = (\rho e) v_2, \qquad \rho v = (\rho e) v_3, \qquad E = (\rho e)\LRp{1 - \frac{{v_2^2+v_3^2}}{2 v_4}},
\end{align*}
where $\rho e$ and $s$ can be expressed in terms of the entropy variables as
\begin{equation*}
\rho e = \LRp{\frac{(\gamma-1)}{\LRp{-v_4}^{\gamma}}}^{1/(\gamma-1)}e^{\frac{-s}{\gamma-1}}, \qquad s = \gamma - v_1 + \frac{{v_2^2+v_3^2}}{2v_4}.
\end{equation*}

There exist several choices for entropy conservative fluxes \cite{ismail2009affordable, ranocha2018comparison,chandrashekar2013kinetic}.  We utilize the
the entropy conservative numerical fluxes given by Chandrashekar in \cite{chandrashekar2013kinetic}
\begin{align*}
&f^1_{1,S}(\bm{u}_L,\bm{u}_R) = \avg{\rho}^{\log} \avg{u},& &f^1_{2,S}(\bm{u}_L,\bm{u}_R) = \avg{\rho}^{\log} \avg{v},&\\
&f^2_{1,S}(\bm{u}_L,\bm{u}_R) = f^1_{1,S} \avg{u} + p_{\rm avg},&  &f^2_{2,S}(\bm{u}_L,\bm{u}_R) = f^1_{2,S} \avg{u},&\nonumber\\
&f^3_{1,S}(\bm{u}_L,\bm{u}_R) = f^2_{2,S},& &f^3_{2,S}(\bm{u}_L,\bm{u}_R) = f^1_{2,S} \avg{v} + p_{\rm avg},&\nonumber\\
&f^4_{1,S}(\bm{u}_L,\bm{u}_R) = \LRp{E_{\rm avg} + p_{\rm avg}}\avg{u},& &f^4_{2,S}(\bm{u}_L,\bm{u}_R) = \LRp{E_{\rm avg} + p_{\rm avg} }\avg{v},& \nonumber
\end{align*}
where the quantities $p_{\rm avg}, E_{\rm avg},  \nor{\bm{u}}^2_{\rm avg}$ are defined as
\begin{gather*}
p_{\rm avg} = \frac{\avg{\rho}}{2\avg{\beta}}, \qquad E_{\rm avg} = \frac{\avg{\rho}^{\log}}{2\avg{\beta}^{\log}\LRp{\gamma -1}}   + \frac{\nor{\bm{u}}^2_{\rm avg}}{2}, \qquad  \beta = \frac{\rho}{2p},\\
 \nor{\bm{u}}^2_{\rm avg} = 2(\avg{u}^2 + \avg{v}^2) - \LRp{\avg{u^2} +\avg{v^2}} \nonumber.
\end{gather*}
From here on, \emph{entropy conservative} refers to a scheme which uses these entropy conservative fluxes at inter-element interfaces.  We will refer to schemes which add interface dissipation as \emph{entropy stable}.  In this work, we utilize a local Lax-Friedrichs interface dissipation.  

For all convergence experiments, we compare the numerical solution to analytic solution for the isentropic vortex problem \cite{shu1998essentially} 
\begin{align*}
\rho(\bm{x},t) &= \LRp{1 - \frac{\frac{1}{2}(\gamma-1)(\beta e^{1-r(\bm{x},t)^2})^2}{8\gamma \pi^2}}^{\frac{1}{\gamma-1}}, \qquad p = \rho^{\gamma},\\
u(\bm{x},t) &= 1 - \frac{\beta}{2\pi} e^{1-r(\bm{x},t)^2}(x_2-c_2), \qquad v(\bm{x},t) = \frac{\beta}{2\pi} e^{1-r(\bm{x},t)^2}(x_2-c_2).\nonumber
\end{align*}
Here, $u, v$ are the $x_1$ and $x_2$ velocity and $r(\bm{x},t) = \sqrt{(x_1-c_1-t)^2 + (x_2-c_2)^2}$.  The following experiments utilize $c_1 = 5, c_2 = 0$ and $\beta = 5$.  

A low storage 4th order Runge-Kutta scheme \cite{carpenter1994fourth} is used for all numerical experiments.  The time-step is estimated based on formulas derived for linear advection in \cite{chan2015gpu, chan2018multi}
\[
dt = C \frac{h}{c_{\max} \max\LRc{\frac{1}{2}C_T, C_I}}
\]
where $h$ is the mesh size, $c_{\max}$ is the maximum wave-speed, $C$ is a user-defined CFL constant, and $C_T, C_I$ are constants in finite element inverse and trace inequalities.  These constants scale proportionally to $N^2$, though precise values of $C_I, C_T$ vary slightly depending on the choice of volume or surface quadrature used.  The dependence of $C_I, C_T$ on quadrature is discussed in more detail in Appendix~\ref{sec:consts}, where computed values are also given.

\subsection{Choices of volume and surface quadrature considered}
\label{sec:opts}

On quadrilaterals, we consider volume quadratures which are tensor products of one-dimensional quadrature rules, while for triangles we utilize volume quadratures from \cite{xiao2010quadrature}.  Surface quadratures are constructed face-by-face, and we refer to surface quadrature rules by the specific quadrature used over each face.  

We consider three choices of volume and surface quadrature on quadrilaterals: 
\begin{enumerate}
\item $(N+1)$ point GLL volume quadrature, $(N+1)$ point GLL surface quadrature.
\item $(N+1)$ point GLL volume quadrature, $(N+1)$ point Gauss surface quadrature,
\item $(N+1)$ point Gauss quadrature, $(N+1)$ point Gauss surface quadrature,
\end{enumerate}
On triangles, we consider two cases:
\begin{enumerate}
\item degree $2N$ volume quadrature, $(N+1)$ point Gauss surface quadrature,
\item degree $2N$ volume quadrature, $(N+1)$ point GLL surface quadrature.
\end{enumerate}

These choices can be combined to provide three different options on two-dimensional hybrid meshes of quadrilateral and triangular elements, which are motivated by balancing computational efficiency and accuracy:
\begin{enumerate}
\item Option 1: $(N+1)$ point GLL volume quadrature on quadrilaterals and $(N+1)$ point GLL surface quadrature on quadrilaterals and triangles.\label{opt:1}
\item Option 2: $(N+1)$ point GLL volume quadrature on quadrilaterals and $(N+1)$ point Gauss surface quadrature on quadrilaterals and triangles.\label{opt:2}
\item Option 3: $(N+1)$ point Gauss volume quadrature on quadrilaterals and $(N+1)$ point Gauss surface quadrature on quadrilaterals and triangles.\label{opt:3}
\end{enumerate}
All three options assume a triangular volume quadrature which is exact for all polynomials of degree $2N$ or less.  

All three options result in similar computational costs on triangles, but slight differences in computational cost and complexity on quadrilaterals.  On quadrilaterals, Option~\ref{opt:1} is the most computationally efficient option, as the formulation (\ref{thm:skewformcurved}) reduces to a standard entropy stable DG-SEM \cite{gassner2016split} or spectral collocation method \cite{carpenter2014entropy}.  Option 3 is most involved, resulting in a Gauss collocation method on quadrilaterals \cite{chan2018efficient}, and requires interpolation and two-point flux interactions between lines of volume quadrature nodes and boundary nodes.  

Option~\ref{opt:2} is slightly more expensive than Option~\ref{opt:1}, as the solution must be interpolated from GLL to Gauss nodes on the boundary.  However, this is less expensive than Option~\ref{opt:3} since volume GLL nodes include GLL boundary nodes as a subset.  This implies that the matrix $\bm{V}_f$ is sparse, such that interpolation to boundary Gauss nodes involves only nodal values at boundary GLL nodes.  Thus, Option~\ref{opt:2} requires only two-point flux computations between boundary GLL and Gauss nodes.  In contrast, the Gauss collocation scheme in Option~\ref{opt:3} computes two-point flux interactions through $\bm{f}_S$ between each boundary node and a line of volume nodes.  


\subsection{Verification of discrete entropy conservation}

We first verify that, for an entropy conservative flux and periodic domain, the spatial formulation tested against the projected entropy variables is numerically zero.  We refer to this quantity as the entropy right-hand side (RHS).   Section~\ref{sec:quadacc} outlines conditions on quadrature accuracy which guarantee that the formulations (\ref{eq:esdgSkew}) and (\ref{eq:skewformcurved}) are discretely entropy conservative.  These numerical results confirm that these conditions are tight.  

We induce a curved polynomial mapping by defining curved coordinates $\tilde{\bm{x}}$ through a mapping of Cartesian coordinates $\bm{x} \in [-1,1]^2$ 
\begin{align*}
\tilde{x}_1 &= x_1 + \alpha \cos\LRp{\frac{\pi}{2}x_1}\sin\LRp{\pi x_2}\\
\tilde{x}_2 &= x_2 + \alpha \sin\LRp{{\pi}x_1}\cos\LRp{\frac{\pi}{2} x_2},
\end{align*}
where $\alpha = 1/8$ for the following experiments.  We vary the geometric degree of this mapping from $N_{\rm geo} = 1$ to $N_{\rm geo} = N$, where $N_{\rm geo}$ denotes the polynomial degree of the geometric mapping on a quadrilateral or triangular element.  

Since Assumption~\ref{ass:quad} requires that the volume quadrature is sufficiently accurate to ensure that the mass matrix is positive-definite, we fix the volume to quadrature to be exact for degree $2N$ polynomials on triangles.  On quadrilaterals, we fix the volume quadrature to be an $(N+1)$ point GLL quadrature.  To verify the conditions given in Section~\ref{sec:quadacc}, we vary the accuracy of the 2D surface quadrature rule.  

The initial condition is taken as the $L^2$ projection of the discontinuous profile
\[
    \rho = \begin{cases}
    3 & |x-x_0| < 2.5\\
    2 & \text{ otherwise}
    \end{cases}, \qquad 
    x_0 = 7.5, \qquad u = v = 0, \qquad p = \rho^\gamma.
\]
We evolve the solution until final time $T = 1.0$ on a domain $[0,15]\times[-.5,.5]$ using the skew symmetric formulation with $N = 6$ and a CFL of $1/2$.  Table~\ref{table:ecrhs} shows the maximum entropy RHS over the duration of the simulation. We observe that for all $N_{\rm geo} \leq M+1$, the maximum entropy RHS is $O(10^{-14})$ and at the level of machine precision.  When $N_{\rm geo}> M+1$, we observe that the maximum entropy RHS increases.  We note that the case of $M= 5$ for the quadrilateral corresponds to the use of an $(N+1)$-point GLL rule for both volume and surface quadrature.  For this choice of quadrature, the skew-symmetric formulation is equivalent to the entropy stable spectral collocation or DG-SEM  methods of \cite{carpenter2014entropy, gassner2016split}.  

\begin{table}
\centering
\subfloat[Triangular mesh]{
\begin{tabular}{c||c|c|c|c|c|c|}
& $N_{\rm geo} = 1$ & $N_{\rm geo} = 2$ & $N_{\rm geo} = 3$ & $N_{\rm geo} = 4$ & $N_{\rm geo} = 5$ & $N_{\rm geo} = 6$ \\
\hline\hline
$M = 5 $ &  8.68e-14 & 9.41e-14 & 9.31e-14 & 9.10e-14& 9.92e-14 & 8.90e-14\\
\hline
$M = 3 $ &  1.01e-13 & 8.87e-14 & 8.79e-14 & 9.68e-14 & 0.00833 & 0.00967 \\
\hline
$M = 1 $ & 1.8e-13 & 1.82e-13 & 1.998 & 2.104 & 2.080 & 2.086\\
\hline
\end{tabular}
}\\
\subfloat[Quadrilateral mesh]{
\begin{tabular}{c||c|c|c|c|c|c|}
& $N_{\rm geo} = 1$ & $N_{\rm geo} = 2$ & $N_{\rm geo} = 3$ & $N_{\rm geo} = 4$ & $N_{\rm geo} = 5$ & $N_{\rm geo} = 6$ \\
\hline\hline
$M = 5$ & 1.74e-14 & 1.06e-14 & 1.31e-14 &  1.35e-14 & 1.06e-14 & 1.31e-14\\
\hline
$M = 3$ & 2.89e-14 & 2.59e-14 &  2.83e-14 & 2.43e-14 & 3.20e-05 & 3.27e-05 \\
\hline
$M = 1$ &4.68e-14 & 3.77e-14 & 0.1548 &  0.1532 & 0.1517 & 0.1517\\
\hline
\end{tabular}
}
\caption{Maximum absolute value of the entropy RHS for degree $N=6$ over $t\in [0,1]$ on triangular and quadrilateral meshes.  The volume quadrature for the quadrilateral mesh is taken to be $(N+1)$-point GLL quadrature.  The surface quadrature is taken to be a 1D GLL quadrature with a varying number of points, such that the rule is exact for degree $M+N$ polynomials.  }
\label{table:ecrhs}
\end{table}


\subsection{Hybrid quadrilateral-triangular meshes}

We conclude with experiments on a mixed mesh containing both quadrilateral and triangular elements (Figure~\ref{fig:hybridmesh}).  Figure~\ref{fig:hybriderrors} shows $L^2$ errors for the isentropic vortex computed at $T = 5$ for Options~\ref{opt:1}, \ref{opt:2}, and \ref{opt:3}.  

We observe that, in all cases, Option~\ref{opt:1} is less accurate than Options~\ref{opt:2} and \ref{opt:3}, and that Option~\ref{opt:3} achieves a rate of convergence close to the optimal $O(h^{N+1})$ rate, while Option~\ref{opt:1} typically achieves rates of convergence between $O(h^N)$ and $O(h^{N+1/2})$.  However, Option~\ref{opt:2} behaves differently depending on the order $N$.  For $N = 1$, Option~\ref{opt:2} achieves an accuracy similar to Option~\ref{opt:1}.  However, as $N$ increases, Option~\ref{opt:2} becomes more accurate.  For $N=4$, Option~\ref{opt:2} achieves the same level of error observed for Option~\ref{opt:3}.  This suggests that Lemma~\ref{lemma:dsbpapprox} may be sufficient but not necessary for full order accuracy.  These results may also differ depending on the type of interface dissipation applied \cite{hindenlang2019order}.


\begin{figure}
\centering
\subfloat{\raisebox{3em}{
\includegraphics[width=.43\textwidth]{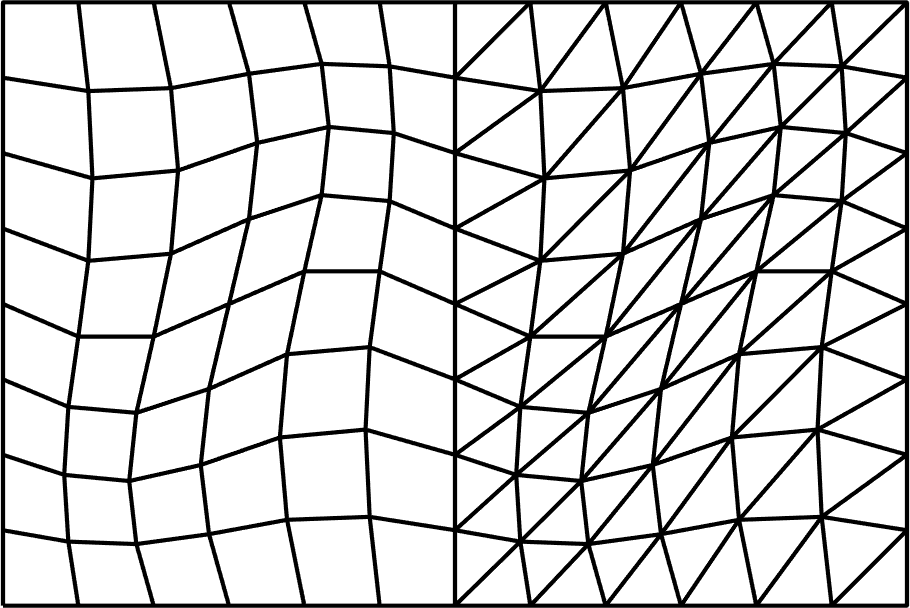}
}\label{fig:hybridmesh}}
\subfloat[Convergence for $N = 1,2,3,4$]{
\begin{tikzpicture}
\begin{loglogaxis}[
    width=.5\textwidth,
    xlabel={Mesh size $h$},
    ylabel={$L^2$ errors}, 
    xmax=.75,
    ymin=1e-5, ymax=5,
    legend pos=south east, legend cell align=left, legend style={font=\tiny},	
    xmajorgrids=true, ymajorgrids=true, grid style=dashed,
    legend entries={GLL-GLL, GLL-Gauss, Gauss-Gauss} 
]
\pgfplotsset{
cycle list={{blue, mark=*}, {red, dashed ,mark=square*},{black ,mark=triangle*}}
}
\addplot+[semithick, mark options={solid, fill=markercolor}]
coordinates{(0.333333,2.50656)(0.166667,1.874)(0.0833333,1.05647)(0.0416667,0.438926)};
\addplot+[semithick, mark options={solid, fill=markercolor}]
coordinates{(0.333333,2.42633)(0.166667,1.75142)(0.0833333,0.958902)(0.0416667,0.462137)};
\addplot+[semithick, mark options={solid, fill=markercolor}]
coordinates{(0.333333,2.09363)(0.166667,0.981664)(0.0833333,0.264928)(0.0416667,0.0536376)};
\addplot+[semithick, mark options={solid, fill=markercolor}]
coordinates{(0.333333,1.36947)(0.166667,0.390076)(0.0833333,0.0623387)(0.0416667,0.011361)};
\addplot+[semithick, mark options={solid, fill=markercolor}]
coordinates{(0.333333,1.28187)(0.166667,0.326338)(0.0833333,0.0427717)(0.0416667,0.00402484)};
\addplot+[semithick, mark options={solid, fill=markercolor}]
coordinates{(0.333333,0.859449)(0.166667,0.132239)(0.0833333,0.0187852)(0.0416667,0.0026063)};
\addplot+[semithick, mark options={solid, fill=markercolor}]
coordinates{(0.333333,0.48545)(0.166667,0.0615538)(0.0833333,0.00665061)(0.0416667,0.000711264)};
\addplot+[semithick, mark options={solid, fill=markercolor}]
coordinates{(0.333333,0.446407)(0.166667,0.0477486)(0.0833333,0.00372164)(0.0416667,0.000309705)};
\addplot+[semithick, mark options={solid, fill=markercolor}]
coordinates{(0.333333,0.286258)(0.166667,0.0331728)(0.0833333,0.00239593)(0.0416667,0.000144346)};
\addplot+[semithick, mark options={solid, fill=markercolor}]
coordinates{(0.333333,0.211329)(0.166667,0.0179798)(0.0833333,0.000816429)(0.0416667, 4.583068530777824e-05)};
\addplot+[semithick, mark options={solid, fill=markercolor}]
coordinates{(0.333333,0.160937)(0.166667,0.0110045)(0.0833333,0.00041413)(0.0416667, 1.438708299584088e-05)};
\addplot+[semithick, mark options={solid, fill=markercolor}]
coordinates{(0.333333,0.13975)(0.166667,0.00977304)(0.0833333,0.00037731)(0.0416667, 1.337566281886131e-05)};
\node at (axis cs:.5,2.5) {$N = 1$};
\node at (axis cs:.5,1.2) {$N = 2$};
\node at (axis cs:.5,.425) {$N = 3$};
\node at (axis cs:.5,.185) {$N = 4$};
\end{loglogaxis}
\end{tikzpicture}
}
\caption{Coarse hybrid mesh and  $L^2$ errors for the isentropic vortex solution for Option~\ref{opt:1}, Option~\ref{opt:2}, and Option~\ref{opt:3} for $N = 1,\ldots, 4$.}
\label{fig:hybriderrors}
\end{figure}
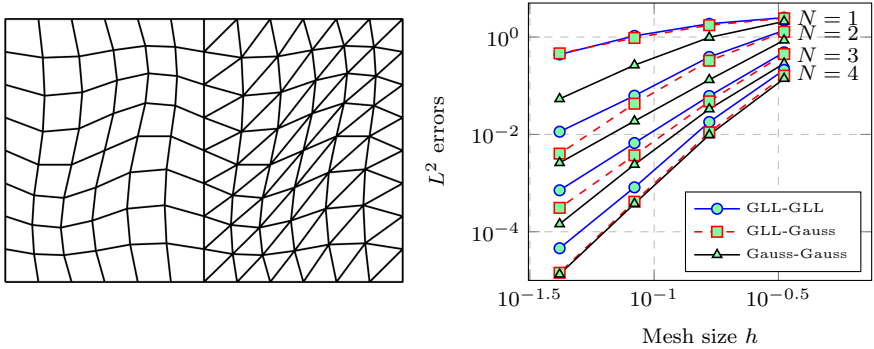

\section{Conclusions}

We have constructed skew-symmetric ``modal'' DG formulations of nonlinear conservation laws which are entropy stable under less restrictive conditions on quadrature accuracy.  These formulations are motivated by volume and surface quadratures which arise naturally on hybrid meshes.  \bnote{Because these quadrature rules do not induce operators which satisfy properties necessary for entropy stability, we derive new ``skew-symmetric'' operators which satisfy necessary conditions under reduced restrictions on quadrature.   We also derive a separate set of conditions relating the accuracy of the new operators and the degree of accuracy of each quadrature rule, and show that design order accuracy is recovered under the common assumptions of degree $2N-1$ volume quadratures and degree $2N$ surface quadratures. } Finally, we derive conditions under which the skew-symmetric formulation is entropy stable on curved meshes in terms of the degree of quadrature accuracy and polynomial degree of the geometric mapping.  Numerical experiments confirm the entropy stability and high order accuracy of the proposed schemes on triangular, quadrilateral, and 2D hybrid meshes.  

\section{Acknowledgments} 

The author thanks David C.\ Del Rey Fernandez for helpful discussions, \rnote{as well as the two anonymous reviewers whose comments significantly improved the readabilty of this manuscript}.  Jesse Chan is supported by the National Science Foundation under awards DMS-1719818 and DMS-1712639.  

\appendix
\section{Dependence of inverse and trace constants on quadrature}
\label{sec:consts}

The maximum stable timestep under explicit time-stepping depends on specific choices of volume and surface quadrature.  The dependence of timestep on quadrature has been documented for tensor product elements in \cite{gassner2011comparison}, where they showed that for a high order Taylor method in time, the maximum stable timestep under  $(N+1)$-point GLL volume and surface quadratures is roughly twice as large as the maximum stable timestep when volume/surface integrals are approximated using $(N+1)$ point Gauss quadratures.  

This discrepancy can be understood in terms of constants in finite element inverse and trace inequalities.  It was shown in \cite{chan2015gpu, chan2018multi} that, for linear problems, the maximum stable time-step scales inversely with the order-dependent constants $C_I, C_T$, where
\begin{align}
\int_{\hat{D}} \LRb{\Grad u}^2 \leq C_I \int_{\hat{D}} u^2, \qquad \int_{\partial \hat{D}} u^2 \leq C_T  \int_{\hat{D}} u^2, \qquad \forall u \in V^N \label{eq:consts}.
\end{align}
Here, the integrals over $\hat{D}, \partial \hat{D}$ are computed using the same volume and surface quadrature rules used for computations.  These constants can be used to bound surface integrals which appear in DG formulations, which can in turn be used to construct bounds on the spectral radius of DG discretization matrices.  The maximum stable timestep $dt_{\max}$ is thus inversely proportional to the inverse and trace constants
\[
dt_{\max} \propto C_T^{-1}, C_I^{-1}
\]
The constants $C_I, C_T$ depend on the choices of volume and surface quadrature used to evaluate each of the integrals in (\ref{eq:consts}).  It is known that $L^2$ norm computed using GLL quadrature is weaker than the full $L^2$ norm \cite{quarteroni1994introduction,canuto2007spectral}.  For the domain $\hat{D} = [-1,1]^d$ in $d$ dimensions, it can be shown that
\begin{equation}
\int_{\hat{D}} u^2 \leq \int_{\hat{D}, {\rm GLL}} u^2 \leq \LRp{2+\frac{1}{N}}^{d/2}\int_{\hat{D}} u^2 \qquad \forall u \in Q^{N},
\label{eq:gllineq}
\end{equation}
where the middle integral is under-integrated using GLL quadrature.  In other words, the discrete $L^2$ norm induced using GLL quadrature is weaker than the $L^2$ norm induced by a more accurate quadrature rule, which will be reflected in the trace and inverse constants.  

\begin{table}
\centering
\subfloat[$C_I$, quadrilateral elements]{
 \begin{tabular}{|c || c| c| c| c| c| c| c|} 
 \hline
 N & 1 & 2& 3 & 4 & 5 & 6& 7\\
 \hline\hline
Volume GLL & 2 & 12 & 37.16 & 91.67 &195.98&374.78&657.28\\
 \hline
Volume Gauss &6&30& 85.06 & 190.12 & 369.45 &652.30 &1072.75\\
 \hline
 \end{tabular}
 }\\
 \subfloat[$C_T$, quadrilateral elements]{
 \begin{tabular}{|c || c| c| c| c| c| c| c|} 
 \hline
 N & 1 & 2& 3 & 4 & 5 & 6& 7\\
 \hline\hline
Volume GLL, surface GLL & 2 & 6 & 12 & 20&30&42&56\\
 \hline
Volume Gauss, surface Gauss &6&12&20&30&42&56&72\\
 \hline
 \end{tabular}
 }\\
 \subfloat[$C_I$, triangular elements]{
 \begin{tabular}{|c || c| c| c| c| c| c| c|} 
 \hline
 N & 1 & 2& 3 & 4 & 5 & 6& 7\\
 \hline\hline
Deg.\ $2N$ vol.\ quad.\  \cite{xiao2010quadrature} &9&39.27&100.10 &213.28 &401.16 &695.48 &1127.48\\
 \hline
 \end{tabular}
 }\\
\subfloat[$C_T$, triangular elements]{
 \begin{tabular}{|c || c| c| c| c| c| c| c|} 
 \hline
 N & 1 & 2& 3 & 4 & 5 & 6& 7\\
 \hline\hline
Surface GLL & 12 & 16.14 & 20.52 & 28.12 &35.42 & 45.97 & 55.76\\
 \hline
Surface Gauss  &6&10.90&16.29& 24 & 31.88 &42.42 &52.89\\
 \hline
 \end{tabular}
 }
 \caption{Inverse and trace constants for triangular and quadrilateral elements with different quadrature configurations.}
 \label{tab:consts}
\end{table}

Table~\ref{tab:consts} shows trace and inverse constants for triangular and quadrilateral elements under several different configurations of quadrature.  For quadrilateral elements at high orders, we observe that the degree $N$ inverse constants $C_I$ under Gauss quadrature are roughly as large as the degree $(N+1)$ inverse constants under (volume) GLL quadrature.  The degree $N$ trace constants $C_T$ under Gauss quadrature are exactly equal to the degree $(N+1)$ trace constants under GLL quadrature, which was proven in \cite{chan2015gpu}.  Trace constants under GLL volume and Gauss surface quadrature are also identical to trace constants computed using GLL quadrature for both volume and surface integrals, which is a consequence of the lower bound in (\ref{eq:gllineq}).  

Several observations can be made based on the values of $C_I, C_T$ presented in Table~\ref{tab:consts}.  On quadrilaterals, the maximum stable time-step for a degree $N$ DG scheme using Gauss quadrature is expected to be smaller than that of a degree $N$ scheme using GLL quadrature, which matches observations in \cite{gassner2011comparison}.  Additionally, the maximum stable timestep under GLL volume quadrature and Gauss surface quadrature should be the same as the maximum stable timestep when GLL quadrature is used for both volume and surface integrals (e.g.\ DG-SEM).  For triangles, the maximum stable timestep should be smaller under surface GLL quadrature compared to surface Gauss quadrature.  However, we note that, while bounds on the maximum stable time-step can be derived based on the constants $C_I, C_T$ \cite{chan2015gpu, chan2018multi}, these bounds are not tight for upwind or dissipative fluxes \cite{krivodonova2013analysis}.  

%

\bibliographystyle{unsrt}
\bibliography{dg}

\end{document}